\documentclass[11pt]{article}
\newif\ifarxiv
\usepackage[activeacute,english]{babel} 
\usepackage[utf8]{inputenc}
\usepackage{verbatim} 
\usepackage{graphicx} 
\usepackage[usenames,dvipsnames]{color} 
\usepackage{enumerate,mdwlist}  
\usepackage{multirow} 
\usepackage{amsmath,amsfonts,amssymb,amsthm} 
\usepackage{xfrac}
\usepackage{empheq} 
\usepackage{bbm,dsfont} 
\usepackage{mathrsfs} 
\numberwithin{equation}{section} 
\usepackage{float} 
\usepackage[labelfont=bf,labelsep=space]{caption} 
\usepackage{empheq} 
\usepackage{yfonts}

\usepackage[OT2,OT1]{fontenc}
\DeclareRobustCommand\cyr{%
  \renewcommand\rmdefault{wncyr}%
  \renewcommand\sfdefault{wncyss}%
  \renewcommand\encodingdefault{OT2}%
  \normalfont
  \selectfont}
\DeclareTextFontCommand{\textcyr}{\cyr}

\usepackage{subcaption}

\usepackage{geometry}
\usepackage{color}
\definecolor{red}{rgb}{.7,0,0}
\definecolor{blue}{rgb}{0,0,1}

\def\yuproj{\textrm{\cyr Yu}}

\def\bbR{\mathbb{R}}

\def\bbP{\mathbb{P}}
\def\bbS{\mathbb{S}}
\def\bbE{\mathop{\mathbb{E}}}

\def\fkf{\mathfrak{f}}
\def\fkF{\mathfrak{F}}

\def\fkh{\mathfrak{h}}

\def\fkt{\mathfrak{t}}

\def\fkx{\mathfrak{x}}
\def\fky{\mathfrak{y}}
\def\fkz{\mathfrak{z}}

\newcommand{\K}[1]{\mathbb{#1}}

\def\Tg{\mathrm{T}}
\def\diff{\mathrm{D}}

\def\dist{\mathrm{dist}}
\DeclareMathOperator{\vol}{vol}

\newcommand{\Vector}[1]{\ensuremath{\boldsymbol{#1}}}
\newcommand{\Sym}{\mathrm{\normalfont{Sym}}}
\newcommand{\PP}{\mathrm{\normalfont{P}}}
\newcommand{\LL}{L^2(\mathbb{S}^{n-1})}

\usepackage{ifpdf}
\ifpdf
\usepackage[pdfencoding=auto,unicode]{hyperref} 
\hypersetup{
 pdftoolbar=true,
 pdfmenubar=true,
 pdfstartview={FitV},
 pdftitle={Probabilistic bounds on best rank-one approximation ratio},
 pdfauthor={Khazhgali Kozhasov, Josu\'{e} Tonelli-Cueto},
 pdfsubject={tensors},
 pdfkeywords={tensors,rank-one approximation}, 
 pdfproducer={overleaf},
 linkcolor=red     
 citecolor=green   
 urlcolor=cyan     
 filecolor=magenta 
}
\fi
\usepackage{bookmark}

\title{Probabilistic bounds on best rank-one approximation ratio}
\author{
Khazhgali Kozhasov\\
Institut für Mathematik\\
Universität Osnabr\"uck\\
Osnabr\"uck, GERMANY\\
{\tt \small khazhgali.kozhasov@uni-osnabrueck.de} 
\and
Josu\'{e} Tonelli-Cueto\thanks{Supported by a postdoctoral fellowship of the 2020 ``Interaction'' program of the \emph{Fondation Sciences Mathématiques de Paris}. Partially supported by the ANR JCJC
GALOP (ANR-17-CE40-0009), the PGMO grant ALMA, and the PHC GRAPE.}\\
Inria Paris \& IMJ-PRG\\ 
Sorbonne Université\\
Paris, FRANCE\\
{\tt \small josue.tonelli.cueto@bizkaia.eu}
}
\date{}

\makeatletter
\def\th@plain{%
  \thm@notefont{}
  \slshape 
}
\def\th@definition{%
  \thm@notefont{}
  \normalfont 
}
\makeatother
\theoremstyle{plain}
\newtheorem{lem}{Lemma}[section]
\newtheorem{prop}[lem]{Proposition}
\newtheorem{theo}[lem]{Theorem}
\newtheorem*{theo*}{Theorem}
\newtheorem*{temptheo*}{Template Theorem}

\theoremstyle{definition}

\theoremstyle{remark}

\newtheorem{remark}[lem]{Remark}

\usepackage{multicol}
\usepackage[ruled,algosection,vlined]{algorithm2e}
\SetKwInOut{postcondition}{Postcondition}
\SetKwInOut{precondition}{Precondition}
\makeatletter
\let\original@algocf@latexcaption\algocf@latexcaption
\long\def\algocf@latexcaption#1[#2]{%
  \@ifundefined{NR@gettitle}{%
    \def\@currentlabelname{#2}%
  }{%
    \NR@gettitle{#2}%
  }%
  \original@algocf@latexcaption{#1}[{#2}]%
}
\makeatother

\begin{document}
\maketitle
\begin{abstract}
We provide new upper and lower bounds on the minimum possible ratio of the spectral and Frobenius norms of a (partially) symmetric tensor. In the particular case of general tensors our result recovers a known upper bound. For symmetric tensors our upper bound unveils that the ratio of norms has the same order of magnitude as the trivial lower bound $1/n^{\frac{d-1}{2}}$, when the order of a tensor $d$ is fixed and the dimension of the underlying vector space $n$ tends to infinity. However, when $n$ is fixed and $d$ tends to infinity, our lower bound is better than $1/n^{\frac{d-1}{2}}$.
\end{abstract}

\section{Introduction}

Representation of data sets in compact and simple formats is an important problem of data science with numerous applications. Vectors, matrices and, more generally, tensors are used to naturally model data points. It is often necessary to retain only some key properties of a data set, that corresponds to an approximation of a tensor by another one with a simpler structure. There are several different models, based on tensor decompositions, that are used for this purpose, see \cite{KoBa, FrTa} and references therein. 
An important special case is an approximation of a given ``data"-tensor with a \emph{rank-one tensor}, see \cite{FrMePaSu}.

For $\K{K}=\K{R}$ or $\K{K}=\K{C}$ let $\K{K}^{\mathbf{n}}=\K{K}^{n_1}\otimes \dots\otimes \K{K}^{n_d}$ denote the $\K{K}$-vector space of \emph{$\mathbf{n}$-tensors} with $\mathbf{n}=(n_1,\dots, n_d)$.
A natural way to measure distance between tensors is given by the norm associated to \emph{the Frobenius (also known as Hilbert-Schmidt) product}, which is defined by the formula
\begin{align}\label{eq:inner_product}
\langle T,T'\rangle\ :=\ \sum_{i_j=1}^{n_j} \overline{t_{i_1\dots i_d}} t_{i_1\dots i_d}',\quad T\ =\ (t_{i_1\dots i_d}),\ T'\ =\ (t'_{i_1\dots i_d}).        
\end{align}
A tensor $T=(t_{i_1\dots i_d})$ is said to be \emph{of rank one}, if there exist unit vectors $\Vector{x}^j\in \bbS(\K{K}^{n_j})$ and a scalar $\lambda\in\K{K}$ such that $t_{i_1\dots i_d}=\lambda x^1_{i_1}\dots x^d_{i_d}$. In this case we write $T=\lambda \Vector{x}^1\otimes \dots\otimes \Vector{x}^d$. 
\emph{The problem of best rank-one approximation} of a tensor $T$ consists in finding a closest rank-one tensor to $T$, i.e.,
\begin{align}
    \underset{\lambda \in \K{K},\ \Vector{x}^j\in \bbS(\K{K}^{n_j})}{\mathrm{\normalfont{min}}}\ \Vert T-\lambda\Vector{x}^1\otimes \dots\otimes \Vector{x}^d\Vert,
\end{align}
where $\Vert \cdot\Vert:=\sqrt{\langle \cdot,\cdot\rangle}$ is \emph{the Frobenius norm} of a tensor. 
This problem is essentially equivalent (see \eqref{eq:error}) to computing \emph{the spectral norm} of $T$, 
\begin{align}\label{eq:spectral}
    \Vert T\Vert_\infty\ :=\ \max_{\Vector{x}^j\in \bbS(\K{K}^{n_j})} \vert \langle T,\Vector{x}^1\otimes\dots\otimes \Vector{x}^d\rangle \vert,
    \end{align}
    and is known to be NP-hard \cite[Thm. 1.13]{HiLi}. 
If $\lambda\Vector{x}^1\otimes \dots\otimes \Vector{x}^d$ is a best rank-one approximation of $T$, then (the square of) \emph{the relative best rank-one approximation error} equals (see, e.g., \cite[Thm. 2.19]{QiLuo})
\begin{align}\label{eq:error}
    \frac{\Vert T-\lambda\Vector{x}^1\otimes\dots\otimes\Vector{x}^d\Vert^2}{\Vert T\Vert^2}\ =\ 1-\frac{\ \Vert T\Vert_\infty^2}{\Vert T\Vert^2}.
\end{align}
The smallest possible ratio of the spectral and the Frobenius norms
\begin{align}\label{eq:BROA}
    \mathcal{A}(\K{K}^{\mathbf{n}})\ :=\ \min_{T\in \K{K}^\mathbf{n}} \frac{\ \ \Vert T\Vert_\infty}{\Vert T\Vert}
\end{align}
is known as \emph{the best rank-one approximation ratio} of the space $\K{K}^{\mathbf{n}}$ (see \cite{Qi2011} and also \cite{KuPe}).
 Computing $\mathcal{A}(\K{K}^{\,\mathbf{n}})$ is equivalent to finding the largest (worst) relative best rank-one approximation error \eqref{eq:error}. 
Note also that $0<\mathcal{A}(\K{K}^\mathbf{n})\leq 1$ and $\mathcal{A}(\K{K}^{\mathbf{n}})$ is just the largest constant $c>0$ so that $\Vert T\Vert_\infty \geq c\Vert T\Vert$ holds for all $T\in \K{K}^\mathbf{n}$.
The number \eqref{eq:BROA} is an attribute of a tensor space and thus depends only on the underground field $\K{K}$ and dimensions $n_1,\dots, n_d$.
On the application side, the best rank-one approximation ratio governs the convergence rate of \emph{greedy rank-one update algorithms}, see \cite{Qi2011, Usch}.

A tensor $T=(t_{i_1\dots i_d})$ of format $(n,\dots,n)$ is called \emph{symmetric}, if $t_{i_{\sigma_1}\dots i_{\sigma_d}} = t_{i_1\dots i_d}$ holds for any permutation on $d$ elements $\sigma$. 
A best rank-one approximation to a symmetric tensor $T$ can be chosen among symmetric rank-one tensors $\lambda \Vector{x}\otimes \dots\otimes \Vector{x}$, see \cite{Banach}. 
\emph{The best rank-one approximation ratio} of the space $\Sym^d(\K{K}^n)$ of symmetric tensors is defined as
\begin{align}\label{eq:SBROA}
    \mathcal{A}(\Sym^d(\K{K}^n))\ :=\ \min_{T\in \Sym^d(\K{K}^n)} \frac{\ \ \Vert T\Vert_\infty}{\Vert T\Vert}.
\end{align}
Computing $\mathcal{A}(\Sym^d(\K{K}^n))$ is equivalent to finding the largest (worst) relative best rank-one approximation error \eqref{eq:error} among symmetric tensors in $\Sym^d(\K{K}^n)$. 
Also, by definition, one has $1\geq \mathcal{A}(\Sym^d(\K{K}^n))\geq \mathcal{A}(\K{K}^\mathbf{n})\geq 0$ for $\mathbf{n}=(n,\dots, n)$.

Finding explicit values for \eqref{eq:BROA} and \eqref{eq:SBROA} is a beautiful mathematical problem with interesting connections to composition algebras \cite{LNSU2018} and Chebyshev polynomials \cite{AKU}.

\subsection{General tensors}

The exact value of $\mathcal{A}(\K{K}^{\,\mathbf{n}})$ and of $\mathcal{A}(\Sym^d(\K{K}^n))$ remains unknown for most $\mathbf{n}=(n_1,\dots, n_d)$, $d$ and $n$. 
One has a general lower bound (see, e.g., \cite{LNSU2018})
\begin{align}\label{eq:trivial_bound}
    \mathcal{A}(\K{K}^\mathbf{n})\ \geq\ \frac{1}{\sqrt{\min_{j=1,\dots, d} \prod_{i\neq j} n_i}}.
\end{align}
The equality holds only if so called \emph{orthogonal ($\,\K{K}=\K{R}$) or, respectively, unitary ($\,\K{K}=\K{C}$) tensors} exist in the tensor space $\K{K}^\mathbf{n}$ (see \cite{LNSU2018}). For example, if $\K{K}=\K{R}$ and for $\mathbf{n}=(n,\dots,n)$ this happens only if $n=1, 2, 4$ or $8$, which are dimensions of the four composition $\K{R}$-algebras. 
It is known that (at least for $\K{K}=\K{R}$) the bound \eqref{eq:trivial_bound} gives the correct order of magnitude when $d$ is fixed.
Specifically, using probabilistic estimates of the uniform norm of random tensors from \cite{TS2014}, the authors of \cite{LNSU2018} prove that the right inequality in
\begin{align*}
\frac{1}{\sqrt{\min_{j=1,\dots, d} \prod_{i\neq j} n_i}}\ \leq\ \mathcal{A}(\K{R}^\mathbf{n})\ \leq\ \frac{\ \ \Vert T\Vert_\infty}{\Vert T\Vert}\ \leq\ \frac{C \sqrt{d\ln d}}{\sqrt{\min_{j=1,\dots,d} \prod_{i\neq j} n_i}}    
\end{align*}
holds with positive probability in $T$, where $C$ is some constant and the entries of $T$ are independent standard Gaussians. 
With similar techniques it was proven earlier \cite{CKP1999} that 
\begin{align*}
    \frac{1}{n}\ \leq\ \mathcal{A}(\K{C}^{n}\otimes \K{C}^{n}\otimes \K{C}^{n})\ \leq\ \frac{3\sqrt{\pi}}{n}. 
\end{align*}
In this work, we reprove these probabilistic upper bounds giving explicit values for the constant $C$ for tensors of arbitrary order $d$.
\begin{theo}\label{thm:combig}
For any $d\geq 3$ and $\Vector{n}=(n_1,\dots, n_d)$ with $n_1,\dots, n_d\geq 2$ we have
\begin{equation}\label{eq:combig}
     \frac{1}{\sqrt{\min_i \prod_{j\neq i} n_j}}\ \leq\ \mathcal{A}(\K{K}^\mathbf{n})\ \leq\ \frac{10 \sqrt{d\ln d}}{\sqrt{\min_i \prod_{j\neq i} n_j}}.
\end{equation}
\end{theo}

\subsection{Symmetric tensors}

However, our biggest contribution regards symmetric tensors. Unlike the general case, the problem of estimating $\mathcal{A}(\Sym^d(\K{R}^n))$ is largely open. The best known upper bound, obtained by Li and Zhao in \cite[Thm. 5.3]{LiZhao}, concerns real symmetric tensors of order $d=3$: 
\[\mathcal{A}(\Sym^3(\K{R}^n))\ \leq\ \frac{1.5}{n^{\frac{\ln 1.5}{\ln 2}}}\ \leq\ \mathcal{O}\left(n^{-0.584}\right).\]
Our main result (see Theorem~\ref{thm:combi}) stated for $d=3$ improves this bound to the optimal  $\mathcal{A}(\Sym^3(\K{R}^n))=\mathcal{O}\left(1/n\right)$.
For an arbitrary fixed $d\geq 3$, our main theorem shows that the lower bound for general tensors is also optimal (up to a constant) for symmetric tensors of order $d$. Quite surprisingly, Theorem \ref{thm:combi} shows that when $n$ is fixed and $d$ grows, the behaviour of $\mathcal{A}(\Sym^d(\K{R}^n))$ significantly differs from that of $\mathcal{A}(\K{K}^\mathbf{n})$, $\Vector{n}=(n,\dots, n)$. For example, for a fixed $n\geq 3$, it follows from our result that
\begin{align*}
\lim_{d\to\infty}\frac{\mathcal{A}(\Sym^d(\K{R}^n))}{\mathcal{A}(\Sym^d(\K{C}^n))}\ =\
\lim_{d\to\infty}\frac{\mathcal{A}(\K{K}^\mathbf{n})}{\mathcal{A}(\Sym^d(\K{R}^n))}\ =\ 0.
\end{align*}
In \cite[Cor. $1.8$]{AKU} it was shown that the inequality $\mathcal{A}(\Sym^d(\K{R}^n))>1/n^{\frac{d-1}{2}}$, $d>2$, is strict even when  $n=4$ or $n=8$, that is, when $\mathcal{A}(\K{R}^{\mathbf{n}})=1/n^{\frac{d-1}{2}}$, $\mathbf{n}=(n,\dots, n)$. To our knowledge, no general lower bounds on $\mathcal{A}(\Sym^d(\K{K}^n))$ that are different from the trivial bound \eqref{eq:trivial_bound} were known.
In Theorem \ref{thm:combi} we discover a new lower bound on $\mathcal{A}(\Sym^d(\K{R}^n))$ that, when $d$ is large compared to $n$, is much better than \eqref{eq:trivial_bound}.

We now state our main theorem. To our knowledge, all the bounds here are novel.

\begin{theo}\label{thm:combi}
For any $d\geq 3$ and $n\geq 2$  we have
\begin{equation}\label{eq:combi}
\begin{aligned}
      \max\left\{\frac{1}{2^{\frac{d}{2}}}\binom{d+n-1}{d}^{-\frac{1}{2}},\ \frac{1}{n^{\frac{d-1}{2}}}\right\}\ &\leq\ \mathcal{A}(\Sym^d(\K{R}^n))\ \leq\ \frac{6\sqrt{n\ln d}}{2^{\frac{d}{2}}}\binom{d+\frac{n}{2}-1}{d}^{-\frac{1}{2}},\\
      \max\left\{\binom{d+n-1}{d}^{-\frac{1}{2}},\ \frac{1}{n^{\frac{d-1}{2}}}\right\}\ &\leq\ \mathcal{A}(\Sym^d(\K{C}^n))\ \leq\ 10\sqrt{n\ln d}\binom{d+n-1}{d}^{-\frac{1}{2}}.
\end{aligned}
\end{equation}
In particular, we have that
\begin{align}\label{eq:sym_bound}
       \frac{1}{n^{\frac{d-1}{2}}}\ \leq\ \mathcal{A}(\Sym^d(\K{K}^n))\ \leq\ 6\left(1+\frac{1}{\ln d}\right) \sqrt{d!\ln d}\frac{1}{n^{\frac{d-1}{2}}},
\end{align}
and, for $d\geq n^2/4$, we have
\begin{equation}\label{eq:asymptotic}
\begin{aligned}
         \sqrt{\frac{(n-1)!}{2^dd^{n-1}}}\left(1-\frac{n^2}{4d}\right)\ &\leq\ \mathcal{A}(\Sym^d(\K{R}^n))\ \leq\ 9\sqrt{\frac{\left(\frac{n}{2}\right)!\ln d}{2^d d^{\frac{n}{2}-1}}}\left(1+\frac{1}{4d}\right)\\
          \sqrt{\frac{(n-1)!}{d^{n-1}}}\left(1-\frac{n^2}{4d}\right)\ &\leq\ 
           \mathcal{A}(\Sym^d(\K{C}^n))\ \leq\ 10\sqrt{\frac{n!\ln d}{d^{n-1}}}
\end{aligned},
\end{equation}
where $\left(\frac{n}{2}\right)!:=\Gamma\left(\frac{n}{2}+1\right)$ allows for a better and easier comparison of the bounds.
\end{theo}
\begin{remark}
We require that $d\geq 3$. When  $d=2$, we deal with matrices, in which case
\begin{align*}
    \mathcal{A}(\Sym^2(\K{R}^n))\ =\ \mathcal{A}(\Sym^2(\K{C}^n))\ =\ \frac{1}{\sqrt{n}},
    \end{align*}
and the bound for the ratio of norms is reached, for example, for the identity matrix.
\end{remark}

Theorem \ref{thm:combi} in particular implies that $\mathcal{A}(\K{K}^{\mathbf{n}})$ and $\mathcal{A}(\Sym^d(\K{K}^n))$ have the same order of magnitude $n^{-\frac{d-1}{2}}$ when $n$ is large and $d$ is bounded. 
Recently, Cao et al. gave an alternative derivation of this fact using partitioned block tensors, see \cite[Thm. 4.6]{Cao}.



\subsection{Partially symmetric tensors}

In Subsection \ref{sub:BROA} we recall the definition of a partially symmetric tensor as well as of the best rank-one approximation ratio \eqref{eq:PSBROA} of the space $\mathcal{A}\left(\bigotimes_{j=1}^m \Sym^{d_j}(\K{R}^{n_j})\right)$ of all such tensors.
Our methods can be also applied to this case as the following theorem shows. 

\begin{theo}\label{them:combi2}
For $m\geq 1$, $d_1,\ldots,d_m\geq 2$ with $\max_j d_j\geq 3$ and $n_1,\ldots,n_m\geq 2$ we have
\begin{multline}\label{eq:combi2R}
 \max\left\{\frac{1}{2^{\frac{1}{2}\sum_{j=1}^m d_j}}\prod_{j=1}^m\binom{d_j+n_j-1}{d_j}^{-\frac{1}{2}},\ \sqrt{\frac{\max_j n_j}{\prod_{j=1}^m n_j^{d_j}}}\right\}\ \leq\ \mathcal{A}\left(\bigotimes_{j=1}^m \Sym^{d_j}(\K{R}^{n_j})\right)\ \\
 \leq\ \frac{6\sqrt{\left(\sum_{j=1}^m n_j\right)\ln\left(m\max_j d_j\right)}}{2^{\frac{1}{2}\sum_{j=1}^m d_j}}\prod_{j=1}^m \binom{d_j+\frac{n_j}{2}-1}{d_j}^{-\frac{1}{2}}   
\end{multline}
and
\begin{multline}\label{eq:combi2C}
 \max\left\{\prod_{j=1}^m\binom{d_j+n_j-1}{d_j}^{-\frac{1}{2}},\ \sqrt{\frac{\max_j n_j}{\prod_{j=1}^m n_j^{d_j}}}\right\}\  \leq\ \mathcal{A}\left(\bigotimes_{j=1}^m \Sym^{d_j}(\K{C}^{n_j})\right)\ \\
 \leq\ 10\sqrt{\left(\sum_{j=1}^m n_j\right)\ln\left(m\max_j d_j\right)}\prod_{j=1}^m\binom{d_j+n_j-1}{d_j}^{-\frac{1}{2}}.
\end{multline}
\end{theo}

We can see that the bounds for the partially symmetric case are as good as the ones for the symmetric and general cases. More concretely, when $d_1=\dots= d_m=1$ (that is, in the case of general tensors of format $(n_1,\dots, n_m)$) this bound agrees with that of Theorem~\ref{thm:combig} (up to a constant), while when $m=1$ (that is, in the case of symmetric tensors) we recover those of Theorem~\ref{thm:combi}.

\subsection*{Organization} Our main objective is to prove Theorems~\ref{thm:combig} and \ref{thm:combi}, with emphasis on the latter. After introducing the preliminaries in Section~\ref{sec:preliminaries}, we prove the upper bounds in Section~\ref{sec:upper} and the lower bounds in Section~\ref{sec:lower}. Finally, in Section \ref{sec:large-d},  we prove the estimates of Theorem~\ref{thm:combi} for large $d$.

\medskip
\noindent
{\bf Acknowledgments.}\quad
We thank Erik Lundberg for pointing out a reference for \eqref{eq:comparison}. The second author is grateful to Evgenia Lagoda for moral support and Gato Suchen for suggestions regarding the proof of Theorem~\ref{theo:generalboundtheo}. 

\section{Preliminaries}\label{sec:preliminaries}

In this section we state and recall some auxiliary results and facts, as well as define our probabilistic models.

\subsection{Symmetric tensors and homogeneous polynomials}\label{sub:BROA}

The space $\Sym^d(\K{K}^n)$ of symmetric tensors is identified with the space $\PP_{d,n}\simeq \K{K}^N$, where $N:=\binom{d+n-1}{d}$, of $n$-variate homogeneous polynomials (or forms) of degree $d$: 
\begin{align}\label{eq:identification}
    T\in \Sym^d(\K{K}^n)\, \longleftrightarrow\, f\in \PP_{d,n},\quad f(\Vector{x})\ =\ \langle T,\Vector{x}\otimes\dots\otimes \Vector{x}\rangle\ =\ \sum_{i_j=1}^n t_{i_1\dots i_d} x_{i_1}\dots x_{i_d}.
\end{align}
It is convenient to write the form $f$ in the basis of monomials, $f(\Vector{x})=\sum_{\vert \alpha\vert=d} f_\alpha \Vector{x}^\alpha$, where, by symmetry, $f_\alpha = \binom{d}{\alpha} t_{i_1\dots i_d}$ and $\alpha_i$ is the number of $j=1,\dots, d$ with $i_j=i$. 
Under the identification \eqref{eq:identification}, the Frobenius product \eqref{eq:inner_product} is \emph{the Bombieri-Weyl product} of forms,
\begin{align}\label{eq:BW}
    \langle T,T'\rangle\ =\ \langle f, f'\rangle\ :=\ \sum_{\vert\alpha\vert=d} \binom{d}{\alpha}^{-1}\overline{f^{\phantom{\prime}}_\alpha} f_\alpha',\quad T\sim f,\ T\sim f'.
\end{align}
By a result of Banach \cite{Banach}, a best rank-one approximation to a symmetric tensor $T$ can be chosen among symmetric rank-one tensors $\lambda \Vector{x}\otimes \dots\otimes \Vector{x}$. In particular, the spectral norm \eqref{eq:spectral} of $T$ equals the uniform norm $\Vert f\Vert_\infty$ of the restriction of $f$ to the unit sphere $\bbS(\K{K}^n)=\{\Vector{x}\in \K{K}^n\,:\, \Vert\Vector{x}\Vert_2^2=\vert x_1\vert^2+\dots+\vert x_n\vert^2=1\}$,
\begin{align}\label{eq:Banach}
    \Vert T\Vert_\infty\ =\ \max_{\Vector{x}^j\in \bbS(\K{K}^n)} \vert \langle T,\Vector{x}^1\otimes\dots\otimes \Vector{x}^d\rangle \vert\ =\ \max_{\Vector{x}\in \bbS(\K{K}^n)} \vert \langle T,\Vector{x}\otimes\dots\otimes \Vector{x}\rangle \vert\ =:\ \Vert f\Vert_\infty.
\end{align}
The best rank-one approximation ratio of the space $\Sym^d(\K{K}^n)$ is then defined by
\begin{align*}
    \mathcal{A}(\Sym^d(\K{K}^n))\ :=\ \min_{T\in \Sym^d(\K{K}^n)} \frac{\ \ \,\Vert T\Vert_\infty}{\Vert T\Vert}\ =\ \min_{f\in \PP_{d,n}} \frac{\ \ \,\Vert f\Vert_\infty}{\Vert f\Vert},
\end{align*}
where $\Vert f\Vert:=\sqrt{\langle f,f\rangle}$ is the Bombieri-Weyl norm of $f\in \PP_{d,n}$.

A common generalization of spaces $\K{K}^\mathbf{n}$ and $\Sym^d(\K{K}^n)$ is the space $\bigotimes_{j=1}^m \Sym^{d_j}(\K{K}^{n_j})$ of partially symmetric tensors or, equivalently, the space $\PP_{\mathbf{d},\mathbf{n}}\simeq \bigotimes_{j=1}^m \PP_{d_j,n_j}$ of multi-homogeneous polynomials. An element $F$ of $\PP_{\mathbf{d},\mathbf{n}}$ can be written as $$F(\Vector{x}^1,\dots,\Vector{x}^m)\ =\ \sum_{\vert \alpha(j)\vert = d_j} F_{\Vector{\alpha}}\, (\Vector{x}^1)^{\alpha(1)}\cdots(\Vector{x}^m)^{\alpha(m)},$$ where $F_{\Vector{\alpha}}\in \K{K}$, $\Vector{\alpha}=(\alpha(1),\dots,\alpha(m))$, $\vert\alpha(j)\vert=d_j$, are the coefficients of $F$ in the basis of multi-homogeneous monomials. Then the Bombieri-Weyl product and the uniform norm can be defined via

\begin{align}
    \langle F, F'\rangle\ &:=\ \sum_{\vert \alpha(j)\vert=d_j} \binom{\Vector{d}}{\Vector{\alpha}}^{-1} F^{\phantom{\prime}}_{\Vector{\alpha}} F^\prime_{\Vector{\alpha}},\quad \binom{\Vector{d}}{\Vector{\alpha}}=\binom{d_1}{\alpha(1)}\cdots \binom{d_m}{\alpha(m)}\label{eq:product_multihom}\\
    \Vert F\Vert_\infty\ &:=\ \max\limits_{\Vector{x}^j\in \K{S}(\K{K}^{n_j})} \vert F(\Vector{x}^1,\dots,\Vector{x}^m)\vert\label{eq:uniform},
\end{align}
and \emph{the best rank-one approximation of the space $\bigotimes_{j=1}^m \Sym^{d_j}(\K{K}^{n_j})$} is defined by (see \cite{Friedland})
\begin{align}\label{eq:PSBROA}
    \mathcal{A}\left(\bigotimes_{j=1}^m \Sym^{d_j}(\K{K}^{n_j})\right)\ :=\ \min_{F\in \PP_{\mathbf{d},\mathbf{n}}} \frac{\ \ \, \Vert F\Vert_\infty}{ \Vert F\Vert},
\end{align}
where $\Vert F\Vert:=\sqrt{\langle F, F\rangle}$ is given by \eqref{eq:product_multihom}.
Moreover, it is important to keep in mind that the action via changes of variables of the product of unitary groups $U(n_1)\times \dots\times U(n_m)$ on $\mathrm{P}_{\Vector{d},\Vector{n}}$ preserves both the inner product \eqref{eq:product_multihom} and the norm \eqref{eq:uniform}.
In the real case ($\K{K}=\K{R}$) the invariance holds with respect to orthogonal changes of variables.

\subsection{Harmonic polynomials}

A form $h\in \textrm{P}_{d,n}$ is called \emph{harmonic}, if it is annihilated by \emph{the Laplace operator}, that is, $$\frac{\partial^2 h}{\partial x_1^2}+\dots+\frac{\partial^2 h}{\partial x_n^2}\ =\ 0.$$
We denote by $\mathrm{H}_{d,n}\subseteq \textrm{P}_{d,n}$ the subspace consisting of real harmonic $n$-variate forms of degree $d$. 
This space is an irreducible representation of the group $O(n)$ of orthogonal matrices, which acts on $\mathrm{H}_{d,n}$ via change of variables.
By \cite[Sect. 4.5]{Kostlan} any $O(n)$-invariant scalar product on $\mathrm{H}_{d,n}$ is a positive multiple of \emph{the $\LL$-product} defined as
\begin{align}\label{eq:L^2}
    \langle h,h'\rangle_{\LL}\ :=\ \int_{\mathbb{S}^{n-1}} h(\Vector{x}) h'(\Vector{x})\, \textrm{d}\mathbb{S}^{n-1},\quad h, h'\in \mathrm{H}_{d,n},
\end{align}
where $\textrm{d}\mathbb{S}^{n-1}$ is the Riemannian volume measure on the unit sphere $\mathbb{S}^{n-1}=\mathbb{S}(\mathbb{R}^n)$ obtained from its standard embedding in $\K{R}^n$. In particular, this is true for the Bombieri product \eqref{eq:BW} restricted to $\mathrm{H}_{d,n}$. We now relate these two scalar products to each other.
\begin{lem}\label{lem:comparison}
For any $h, h'\in \mathrm{H}_{d,n}$ we have
\begin{align*}
    \langle h,h'\rangle\ =\ \frac{2^{d-1}}{\sqrt{\pi}^{\,n}}\frac{\Gamma\left(d+\frac{n}{2}\right)}{\Gamma\left(d+1\right)}\, \langle h,h'\rangle_{\LL}.
\end{align*}
\end{lem}

\begin{proof}
It is convenient to write the $\LL$-product as
\begin{align}\label{eq:L^2-gaussian}
    \langle h, h'\rangle_{\LL}\ =\ \frac{1}{\sqrt{2}^{\,2d+n-2} \Gamma\left(d+\frac{n}{2}\right)}\int_{\K{R}^n} h(\Vector{x}) h'(\Vector{x}) e^{-\frac{\Vert \Vector{x}\Vert_2^2}{2}}\,\textrm{d}\Vector{x},\quad h, h'\in \mathrm{H}_{d,n}.
\end{align}
Since orthogonally invariant scalar products on $\mathrm{H}_{d,n}$ are all proportional, we can recover the constant of proportionality by looking at $h=h'\in \mathrm{H}_{d,2}\subseteq \mathrm{H}_{d,n}$ defined by
\begin{align*}
h(x_1,x_2)\ :=\ \frac{(x_1+i x_2)^d+(x_1-i x_2)^d}{2}\ =\ r^d\cos (d\theta),\quad x_1=r\cos \theta,\ x_2=r\sin \theta.
\end{align*}
By \cite[Thm. 1.1]{AKU} we have $\Vert h\Vert^2=\langle h,h\rangle=2^{d-1}$. 
Since $h$ depends just on $x_1$ and $x_2$, the representation \eqref{eq:L^2-gaussian} implies that the $\LL$-norm of $h$ satisfies
\begin{equation*}
\begin{aligned}
    \Vert h\Vert^2_{\LL}\ &=\ \frac{\sqrt{2\pi}^{\,n-2}}{\sqrt{2}^{\,2d+n-2} \Gamma\left(d+\frac{n}{2}\right)}\int_{\K{R}^2} h(x_1,x_2)^2 e^{-\frac{x_1^2+x_2^2}{2}}\,\textrm{d}x_1\textrm{d}x_2\\ &=\ \frac{\sqrt{2\pi}^{\,n-2}\sqrt{2}^{\,2d}\Gamma\left(d+1\right)}{\sqrt{2}^{\,2d+n-2} \Gamma\left(d+\frac{n}{2}\right)} \int_0^{2\pi} \cos(d\theta)^2\,\textrm{d}\theta\ =\ \frac{\sqrt{\pi}^{\,n}}{2^{d-1}}\frac{\Gamma\left(d+1\right)}{\Gamma\left(d+\frac{n}{2}\right)} \Vert h\Vert^2,
\end{aligned}
\end{equation*}
which completes the proof.
\end{proof}

The self-duality of $(\mathrm{H}_{d,n},\langle\cdot,\cdot\rangle_{\LL})$ implies that for any $\Vector{x}\in \mathbb{S}^{n-1}$ there is a unique harmonic form $Z_{\Vector{x}}\in \mathrm{H}_{d,n}$ with
\begin{align}\label{eq:zonal}
h(\Vector{x})\ =\ \langle h, Z_{\Vector{x}}\rangle_{\LL} \quad \textrm{for all}\quad h\in \mathrm{H}_{d,n}.
\end{align}
The function $Z_{\Vector{x}}\in \mathrm{H}_{d,n}$ is called \emph{the zonal harmonic with pole $\Vector{x}$}. By \cite[Cor. 2.9]{SteinWeiss} the $\LL$-norm of $Z_{\Vector{x}}$ satisfies
\begin{equation}\label{eq:normZ}
    \Vert Z_{\Vector{x}}\Vert^2_{\LL}\ =\ Z_{\Vector{x}}(\Vector{x})\ =\ \frac{D_{d,n}}{\vert \K{S}^{n-1}\vert},
\end{equation}
where $D_{d,n}=\dim \mathrm{H}_{d,n}$ and  $\vert \K{S}^{n-1}\vert=2\sqrt{\pi}^n/\Gamma\left(\frac{n}{2}\right)$ is the volume of the unit sphere.


\subsection{Probabilistic models}\label{sub:models}

We consider real and complex random polynomials and tensors. 
Recall first that a complex random variable $t$ is called \emph{ standard complex Gaussian}, if its real and imaginary parts are independent centered Gaussians with variance $1/2$. 
Let $V$ be a (complex) inner product space. 
Then \emph{the Gaussian distribution} on $V$ is modelled via the vector $\Vector{v}=\sum_{i=1}^N \fkt_i\Vector{v}_i$, where $\fkt_1,\dots, \fkt_N$ are independent standard (complex) Gaussians and $\Vector{v}^1,\dots, \Vector{v}^N$ form an orthonormal (respectively, unitary) basis of $V$.
A remarkable property of the Gaussian distribution on $V$ is its orthogonal (respectively, unitary) invariance. 
This means that the random vector $U\Vector{v}$ has Gaussian distribution for any orthogonal (unitary) transformation $U$ on $V$.

As spaces $\K{K}^{\Vector{n}}$, $\textrm{P}_{d,n}$, $\textrm{P}_{\Vector{d},\Vector{n}}$ and $\mathrm{H}_{d,n}$ are endowed with inner products \eqref{eq:inner_product}, \eqref{eq:BW}, \eqref{eq:product_multihom} and \eqref{eq:L^2-gaussian} respectively, the Gaussian distribution is naturally defined for each of them.
For example, a real (respectively, complex) tensor $\mathcal{T}=(\fkt_{i_1\dots i_d})$ is \emph{Gaussian}, if its entries $\fkt_{i_1\dots i_d}$ are independent standard  (complex) Gaussians. 
Under the standard action of the product of unitary groups $U(\Vector{n}):=U(n_1)\times\dots\times U(n_d)$ on $\K{C}^{\Vector{n}}$ the inner product \eqref{eq:inner_product} and hence the Gaussian distribution on  $\K{C}^{\Vector{n}}$ are invariant. In particular, the inner product space $\K{R}^{\Vector{n}}$ of real tensors and the Gaussian distribution on it are invariant under the product $O(\Vector{n}):=O(n_1)\times\dots\times O(n_d)\subset U(\Vector{n})$ of orthogonal groups.

The Gaussian distribution on $\textrm{P}_{d,n}$ is also known as \emph{Kostlan distribution}. Specifically, an $n$-variate real (respectively, complex) homogeneous polynomial $\fkf(\Vector{x})=\sum_{\vert\alpha\vert=d} \fkf_{\alpha}\Vector{x}^\alpha$ of degree $d$ is called \emph{Kostlan}, if its normalized coefficients $\fkf_{\alpha}/\sqrt{\binom{d}{\alpha}}$ are independent standard (complex) Gaussians.
Similarly, a multi-homogeneous polynomial $\mathcal{F}(\Vector{x}^1,\dots, \Vector{x}^m)=\sum_{\vert\alpha(j)\vert=d_j}\mathcal{F}_{\Vector{\alpha}}(\Vector{x}^1)^{\alpha(1)}\cdots (\Vector{x}^m)^{\alpha(m)}$ of multi-degree $\Vector{d}=(d_1,\dots, d_m)$ is \emph{Kostlan} (that is, Gaussian), if its normalized coefficients $\mathcal{F}_{\Vector{\alpha}}/\sqrt{\binom{\Vector{d}}{\Vector{\alpha}}}$ are independent standard (complex) Gaussians. The notion of Kostlan (partially) symmetric tensor is then unambiguously defined through isomorphisms $\Sym^d(\K{K}^n)\simeq \textrm{P}_{d,n}$ and $\bigotimes_{j=1}^m\Sym^{d_j}(\K{K}^{n_j})\simeq \textrm{P}_{\mathbf{d},\mathbf{n}}$.
Since the inner products on spaces $\textrm{P}_{d,n}$, $\textrm{P}_{\Vector{d},\Vector{n}}$ and $\mathrm{H}_{d,n}$ are invariant with respect to orthogonal (unitary) changes of variables, so are Gaussian distributions on each of them.








\section{Upper bounds}\label{sec:upper}

In this section we prove upper bounds stated in main Theorems \ref{thm:combig} and \ref{thm:combi}. These are obtained by combining Propositions \ref{cor:bound_general}, \ref{cor:sym} and \ref{cor:harm}, which in turn follow from combining Proposition~\ref{prop:projection} with Theorem \ref{theo:generalboundtheo}. We also prove the bounds of Theorem~\ref{them:combi2} by proving Propositions~\ref{cor:part_symA} and~\ref{cor:part_symB}. In the sequel we assume that $d\geq 3$ and $n, n_1,\dots, n_d\geq 2$.

\subsection{Subgaussian estimates for the projection of a Gaussian vector}

The main result of this subsection is the following estimate of the tail probability of the ratio of norms of two vectors, one of which is the image of the other under an orthogonal projection. In technical terms, we are just proving that $\|P\fkx\|_2/\|\fkx\|_2$ is a subgaussian random variable, see~\cite{vershyninbook} for more details.

\begin{prop}\label{prop:projection}
Let $P:\bbR^N\rightarrow V$ be an orthogonal projection onto a $k$-dimensional subspace $V\subseteq \bbR^N$ and let $\fkx\in\bbR^{N}$ be a standard Gaussian vector. Then the random variable  $\frac{\|P\fkx\|_2}{\|\fkx\|_2}$ satisfies
\begin{align}\label{eq:projection}
\bbP_\fkx\left(\frac{\|P\fkx\|_2}{\|\fkx\|_2}\geq t\right)\ \leq\  3\exp\left(-\frac{N}{3e^{k-1}}t^2\right)\quad \textrm{for all}\quad t\geq 0.
\end{align}
\end{prop}

\begin{remark}The bound Proposition~\ref{prop:projection} is meaningful only if $t\leq 1$, as for $t>1$ we  have $\bbP_\fkx\left(\|P\fkx\|_2/\|\fkx\|_2\geq t\right)=0$.\end{remark}
\begin{remark}\label{rem:projection}
Let $P: \K{C}^N\rightarrow V$ be a \emph{unitary projection} onto a $k$-dimensional (complex) subspace $V\subseteq \K{C}^N$ (that is, $\langle P(\Vector{z}),\Vector{z}-P(\Vector{z})\rangle_2=0$ for all $\Vector{z}\in\K{C}^N$). Then it corresponds to an orthogonal projection between real vector spaces, 
\begin{align*}
P^{\,\K{R}}: \K{R}^{2N}&\rightarrow V^{\mathbb{R}},\\
(\Vector{x},\Vector{y})&\mapsto \left(\Re(P(\Vector{x}+i\Vector{y})),\Im(P(\Vector{x}+i\Vector{y}))\right),
\end{align*}
where the $2k$-dimensional real subspace $V^{\mathbb{R}}\subseteq \K{R}^{2N}$ is \emph{the realification of $V\subseteq \mathbb{C}^N$}. Moreover, if $\fkz=\Vector{\fkx}+i \Vector{\fky} \in \K{C}^N$ is a standard complex Gaussian vector, then $\sqrt{2}\,(\Vector{\fkx},\Vector{\fky})\in \K{R}^{2N}$ is a standard real Gaussian vector and, by Proposition \ref{prop:projection}, we obtain for $t\geq 0$ that
\begin{align*}
    \bbP_\fkz\left(\frac{\|P\fkz\|_2}{\|\fkz\|_2}\geq t\right)\ =\ 
    \bbP_{\Vector{\fkx},\Vector{\fky}}\left(\frac{\|P^{\,\K{R}}(\Vector{\fkx},\Vector{\fky})\|_2}{\|(\Vector{\fkx},\Vector{\fky})\|_2}\geq t\right)\leq\  3\exp\left(-\frac{2N}{3e^{2k-1}}t^2\right).
\end{align*}
We need this ``complex" version of \eqref{eq:projection} to obtain bounds on the ratio of norms in case of complex tensors and forms.
\end{remark}

To prove Proposition~\ref{prop:projection}, we will need the following proposition. This proposition is a variation of \cite[Proposition~4.24]{CETC-norms}, which in turn is \cite[Proposition 2.5.2]{vershyninbook} with explicit constants. We include its proof in the appendix for the sake of completeness. 

\begin{prop}\label{prop:subgaussiantailbound}
Let $\fkt\in\bbR$ be a random variable, $C\geq 1$ and $K\geq 0$.
\begin{enumerate}
    \item If for all even integers $\ell>0$, $\left(\bbE_{\fkt}|\fkt|^\ell\right)^{\frac{1}{\ell}}\leq K\sqrt{\ell}$, then for all $t>0$, $\bbP(|\fkt|\geq t)\leq 3e^{-\frac{t^2}{6K^2}}$.
    \item If for all $t>0$,
    $\bbP(|\fkt|\geq t)\leq Ce^{-\frac{t^2}{K^2}}$,
    then for all $\ell\geq 1$,
    \begin{equation}
        \left(\mathbb{E}_{\fkt}|\fkt|^\ell\right)^{\frac{1}{\ell}}\ \leq\ K\left(\sqrt{\frac{\pi}{2}}+\sqrt{2\ln C}\right)\sqrt{\ell}.
    \end{equation}
    Moreover,
    \begin{equation}\label{eq:specialbound}
        \mathbb{E}_{\fkt}|\fkt|\ \leq\ K\sqrt{2\ln C}\left(1+\frac{1}{\ln C}\right).
    \end{equation}
\end{enumerate}
\end{prop}
\begin{remark}\label{remark:mintrick}
Let $\|\cdot\|_a$ and $\|\cdot\|_b$ be two norms on $\bbR^n$. For a Gaussian random vector $\fkx\in\bbR^N$ we have that
\begin{align*}
    \min_{\Vector{x}\neq 0}\frac{\|\Vector{x}\|_a}{\|\Vector{x}\|_b}\ \leq\ \mathbb{E}_{\fkx}\frac{\|\fkx\|_a}{\|\fkx\|_b}.
\end{align*}
Now, if $\bbP(\|\fkx\|_a/\|\fkx\|_b\geq t)\leq Ce^{-t^2/K^2}$, we obtain, by Proposition~\ref{prop:subgaussiantailbound}, that 
\begin{align}\label{eq:estimate_1945}
    \min_{\Vector{x}\neq 0}\frac{\|\Vector{x}\|_a}{\|\Vector{x}\|_b}\ \leq\ K\sqrt{2\ln C}\left(1+\frac{1}{\ln C}\right).
    \end{align}
    However, we can give a better bound on the minimum of the norm ratio by noticing that
\begin{align*}
    \min_{\Vector{x}\neq 0}\frac{\|\Vector{x}\|_a}{\|\Vector{x}\|_b}\ =\ \inf\left\{t> 0\,\Bigg|\, \bbP\left(\frac{\|\fkx\|_a}{\|\fkx\|_b}\geq t\right)<1\right\}.
    \end{align*}
Since $\bbP(\|\fkx\|_a/\|\fkx\|_b\geq t)\leq e^{\,\ln C-t^2/K^2}$, we obtain
\begin{align}\label{eq:trick}
\min_{\Vector{x}\neq 0}\frac{\|\Vector{x}\|_a}{\|\Vector{x}\|_b}\ \leq\ K\sqrt{\ln C},
\end{align}
which improves the constants in the estimate \eqref{eq:estimate_1945}.
\end{remark}

\begin{proof}[Proof of Proposition~\ref{prop:projection}]
Recall that if $\fky$ and $\fkz$ are independent random variables with $\chi^2$-distribution with, respectively, $k$ and $N-k$ degrees of freedom, then $\frac{\fky}{\fky+\fkz}$ has a $\beta$-distribution with parameters $k/2$ and $(N-k)/2$. In this way, $\frac{\|P\fkx\|_2^2}{\|\fkx\|_2^2}$ has a $\beta$-distribution with parameters $k/2$ and $(N-k)/2$, and so its density  reads
\[
\frac{1}{\beta\left(\frac{k}{2},\frac{N-k}{2}\right)}s^{\frac{k}{2}-1}(1-s)^{\frac{N-k}{2}-1},\quad s\in [0,1].
\]
Doing a change of variables $s= t^2$, we obtain that
\begin{equation*}
 \frac{2}{\beta\left(\frac{k}{2},\frac{N-k}{2}\right)}t^{k-1}(1-t^2)^{\frac{N-k}{2}-1},\quad t\in [0,1],
\end{equation*}
is the density of $\|P\fkx\|_2/\|\fkx\|_2$.
Then a straightforward computation implies that for all $\ell>0$
\begin{equation}\label{eq:moment_beta}
\left(\bbE_\fkx\frac{\|P\fkx\|^\ell_2}{\|\fkx\|^\ell_2}\right)^{\frac{1}{\ell}}\ =\ \left(\frac{\Gamma\left(\frac{k+\ell}{2}\right)}{\Gamma\left(\frac{k}{2}\right)}\frac{\Gamma\left(\frac{N}{2}\right)}{\Gamma\left(\frac{N+\ell}{2}\right)}\right)^{\frac{1}{\ell}}.
\end{equation}
We fix $\ell$ to be a positive even integer. To bound \eqref{eq:moment_beta} we treat each fraction separately. For the first fraction,
\begin{align*}
    \frac{\Gamma\left(\frac{k+\ell}{2}\right)}{\Gamma\left(\frac{k}{2}\right)}\ &=\ \prod_{i=1}^{\frac{\ell}{2}}\left(\frac{k+\ell}{2}-i\right)&(\Gamma(x)=(x-1)\Gamma(x-1))\\
    &\leq\ \left(\frac{k-1}{2}+\frac{\ell}{4}\right)^{\frac{\ell}{2}}&\text{(AM-GM inequality)}\\
    &=\ \left(1+\frac{2(k-1)}{\ell}\right)^{\frac{\ell}{2}}\left(\frac{\ell}{4}\right)^{\frac{\ell}{2}}\leq e^{k-1}\left(\frac{\ell}{4}\right)^{\frac{\ell}{2}}.
\end{align*}
And for the second fraction,
\begin{align*}
    \frac{\Gamma\left(\frac{N}{2}\right)}{\Gamma\left(\frac{N+\ell}{2}\right)}&\ =\  \prod_{i=0}^{\frac{\ell}{2}-1}\frac{1}{\frac{N}{2}+i}\leq \left(\frac{2}{N}\right)^{\frac{\ell}{2}}&(\Gamma(x)=(x-1)\Gamma(x-1))
\end{align*}
We observe that in both bounds it is essential that $\ell/2$ is an integer. Putting the obtained bounds together, we have that for all even integers $\ell>0$
\begin{equation*}
\left(\bbE_\fkx\frac{\|P\fkx\|^\ell_2}{\|\fkx\|^\ell_2}\right)^{\frac{1}{\ell}}\ \leq\ e^{\frac{k-1}{\ell}}\sqrt{\frac{\ell}{4}} \, \sqrt{\frac{2}{N}}\ \leq\ \sqrt{\frac{e^{k-1}}{2N}}\sqrt{\ell}.
\end{equation*}
Hence by Proposition~\ref{prop:subgaussiantailbound} the desired claim follows.
\end{proof}

\subsection{Estimates for a random Lipschitz map on a product of spheres}

We consider the \emph{sum-geodesic distance} on $\prod_{k=1}^d\bbS^{n_k-1}$, given by
\[
\dist_\bbS(\Vector{x},\Vector{y})\ :=\ \sum_{k=1}^d\dist_\bbS(\Vector{x}_k,\Vector{y}_k)\ =\ \sum_{k=1}^d\arccos\langle\Vector{x}_k,\Vector{y}_k\rangle_2
\]
for $\Vector{x},\Vector{y}\in \prod_{k=1}^d\bbS^{n_k-1}$. The following theorem is the main tool for our estimates.

\begin{theo}\label{theo:generalboundtheo}
Let $n_1,\ldots,n_d\geq 2$ and $\fkF:\prod_{k=1}^d \bbS^{n_k-1}\rightarrow [0,\infty)$ be a random Lipschitz function whose Lipschitz constant, $\mathrm{Lip}(\fkF)$, satisfies for some $L\geq 1$,
\begin{equation}\label{eq:lipboundaasumption}
    \mathrm{Lip}(\fkF)\ \leq\ L\max_{\Vector{x}\in \prod_{k=1}^d\bbS^{n_k-1}}\fkF(\Vector{x}).
\end{equation}
Then for all $t>0$, 
\begin{align}\label{eq:bound_Pmax}
    \bbP_\fkF\left(\max_{\Vector{x}\in \prod_{k=1}^d\bbS^{n_k-1}}\fkF(\Vector{x})\geq t\right)\ \leq\ C(L,d;n_1,\ldots,n_d)\max_{\Vector{x}\in \prod_{k=1}^d\bbS^{n_k-1}} \bbP_\fkF\left(\fkF(\Vector{x})\geq \frac{t}{2}\right)
    \end{align}
where $C(L,d;n_1,\ldots,n_d)$ satisfies
\begin{equation*}
    \ln C(L,d;n_1,\ldots,n_d)\ =\ (2+\ln(dL))\left(\sum_{k=1}^d n_k\right)-\frac{1}{2}\sum_{k=1}^d\ln(n_k-1)-d\ln(dL).
\end{equation*}
If  $\fkF$ is invariant under orthogonal changes of variables on $\K{S}^{n_1-1},\dots,\K{S}^{n_d-1}$, the probability $\K{P}_{\fkF}\left(\fkF(\Vector{x})\geq \frac{t}{2}\right)$ does not depend on the point $\Vector{x}\in \prod_{k=1}^d\bbS^{n_k-1}$ and hence we can omit maximum in \eqref{eq:bound_Pmax}.
\end{theo}

\begin{remark}\label{rem:boundsC}
Constant $\ln C(L,d;n_1,\dots, n_d)$ satisfies the following bounds. 
First,
\begin{align*}\ln C(L,d;n_1,\ldots,n_d)\ \geq\ 3d+(1+\ln(dL))\sum_{k=1}^d (n_k-1)\ \geq\ d+\sum_{k=1}^d n_k,
\end{align*}
since $n-\frac{1}{2}\ln(n-1)\geq 2$ for $n\geq 2$ and $1+\ln(dL)\geq 1$. On the other hand,
\begin{align*}\ln C(L,d;n_1,\ldots,n_d)\ \leq\ \left(1+\frac{2}{\ln(dL)}\right)\ln(dL)\left(\sum_{k=1}^d n_k\right)-\ln(dL),\end{align*}
after estimating negative terms.
\end{remark}

\begin{proof}[Proof of Theorem~\ref{theo:generalboundtheo}]
If $\max_{\Vector{x}\in \prod_{k=1}^d\bbS^{n_k-1}}\fkF(\Vector{x})\geq t$, then, by the Lipschitz property and our assumption ~\eqref{eq:lipboundaasumption} on the Lipschitz constant, we have that
\begin{align*}
\left\{\Vector{x}\in\prod_{k=1}^d\bbS^{n_k-1}\,\Bigg|\, \fkF(\Vector{x})\geq \frac{t}{2}\right\}\ \supseteq\ B_\bbS\left(\fkx_\ast,(2L)^{-1}\right),
\end{align*}
where $\fkx_\ast\in \prod_{k=1}^d \bbS^{n_k-1}$ is the maximizer of $\fkF$ and $B_\bbS$ is the ball with respect to the sum-geodesic distance on $\prod_{k=1}^d\bbS^{n_k-1}$ that is centered at $\fkx_\ast$ and has radius $(2L)^{-1}$. In this way, $\max_{\Vector{x}\in \prod_{k=1}^d\bbS^{n_k-1}}\fkF(\Vector{x})\geq t$ implies that for a uniformly sampled $\fkx\in \prod_{k=1}^d\bbS^{n_k-1}$ we obtain
\begin{align*}
\bbP_{\fkx\in \prod_{k=1}^d\bbS^{n_k-1}}\left(\fkF(\fkx)\geq \frac{t}{2}\right)\ \geq\ \frac{\vol\,B_\bbS\left(\fkx_\ast,(2L)^{-1}\right)}{\prod_{k=1}^d\vol\bbS^{n_k-1}}\ =\ \frac{\vol\,B_\bbS\left(\mathbf{e}_1,(2L)^{-1}\right)}{\prod_{k=1}^d\vol\bbS^{n_k-1}},
\end{align*}
where $\mathbf{e}_1:=(\Vector{e}_1,\ldots,\Vector{e}_1)\in\prod_{k=1}^d\bbS^{n_k-1}$. The last equality follows from the fact that the volume of $B_\bbS\left(\Vector{x},r\right)$ is independent of $\Vector{x}\in \prod_{k=1}^d\bbS^{n_k-1}$. Therefore, we have
\begin{align*}
    \bbP_\fkF&\left(\max_{\Vector{x}\in \prod_{k=1}^d\bbS^{n_k-1}}\fkF(\Vector{x})\geq t\right)\\&\leq\ \bbP_\fkF\left(\bbP_{\fkx\in \prod_{k=1}^d\bbS^{n_k-1}}\left(\fkF(\fkx)\geq \frac{t}{2}\right)\geq \frac{\vol\,B_\bbS\left(\mathbf{e}_1,(2L)^{-1}\right)}{\prod_{k=1}^d\vol\bbS^{n_k-1}}\right)&\text{(Implication bound)}\\
    &\leq\ \frac{\prod_{k=1}^d\vol\bbS^{n_k-1}}{\vol\,B_\bbS\left(\mathbf{e}_1,(2L)^{-1}\right)}\bbE_{\fkF} \left[\bbP_{\fkx\in \prod_{k=1}^d\bbS^{n_k-1}}\left(\fkF(\fkx)\geq \frac{t}{2}\right)\right]&\text{(Markov's inequality)}\\
    &=\ \frac{\prod_{k=1}^d\vol\bbS^{n_k-1}}{\vol\,B_\bbS\left(\mathbf{e}_1,(2L)^{-1}\right)}\bbE_{\fkx\in \prod_{k=1}^d\bbS^{n_k-1}}\left[ \bbP_\fkF\left(\fkF(\fkx)\geq \frac{t}{2}\right)\right]&\text{(Tonelli's theorem)}\\
    &\leq\ \frac{\prod_{k=1}^d\vol\bbS^{n_k-1}}{\vol\,B_\bbS\left(\mathbf{e}_1,(2L)^{-1}\right)}\max_{\Vector{x}\in \prod_{k=1}^d\bbS^{n_k-1}} \bbP_\fkF\left(\fkF(\Vector{x})\geq \frac{t}{2}\right).
\end{align*}
It remains to bound $\frac{\prod_{k=1}^d\vol\bbS^{n_k-1}}{\vol\,B_\bbS\left(\mathbf{e}_1,(2L)^{-1}\right)}$. Consider the map
\begin{align*}
    \yuproj:\prod_{k=1}^d\bbR^{n_k-1}&\rightarrow \prod_{k=1}^d\bbS^{n_k-1}\\
    \begin{pmatrix}\Vector{z}_1\\\vdots\\\Vector{z}_d\end{pmatrix}&\mapsto \begin{pmatrix}
    \frac{1}{\sqrt{1+\|\Vector{z}_1\|_2^2}}\begin{pmatrix}1\\\Vector{z}_1\end{pmatrix}\\\vdots\\\frac{1}{\sqrt{1+\|\Vector{z}_d\|_2^2}}\begin{pmatrix}1\\\Vector{z}_d\end{pmatrix}\end{pmatrix}.
\end{align*}
Then, by the result in the appendix \ref{subsec:yuprojjacobian}, we have that
\begin{equation}\label{eq:yuprojjacobian}
    \left|\det\diff_{\Vector{z}}\yuproj\right|=\prod_{k=1}^d\left(1+\|\Vector{z}_k\|^2_2\right)^{-\frac{n_k}{2}},\quad \Vector{z}=(\Vector{z}_1,\dots, \Vector{z}_d)\in \prod_{k=1}^d\bbR^{n_k-1}.
\end{equation}
Now, we do a sequence of changes of variables as follows:
\begin{align*}
    \vol\, B_\bbS\left(\mathbf{e}_1,R\right)&\ \\ 
    &=\ \int\limits_{\sum_{k=1}^d\arctan\|\Vector{z}_k\|\leq R}\,\prod_{k=1}^d\left(1+\|\Vector{z}_k\|^2_2\right)^{-\frac{n_k}{2}}\,\mathrm{d}\Vector{z}_1\cdots\mathrm{d}\Vector{z}_d\\
    &=\ \prod_{k=1}^d\vol\,\bbS^{n_k-2}\int\limits_{\substack{\rho_1,\dots, \rho_d\geq 0,\\\sum_{k=1}^d\arctan\rho_k\leq R}}\,\prod_{k=1}^d\rho_k^{n_k-2}\left(1+\rho_k^2\right)^{-\frac{n_k}{2}}\,\mathrm{d}\rho_1\cdots\mathrm{d}\rho_d\\
    &=\ \prod_{k=1}^d\vol\,\bbS^{n_k-2}\int\limits_{\substack{\phi_1,\dots, \phi_d\geq 0,\\\sum_{k=1}^d\phi_k\leq R}}\,\prod_{k=1}^d(\sin \phi_k)^{n_k-2}\,\mathrm{d}\phi_1\cdots\mathrm{d}\phi_d
    \end{align*}
    \begin{align*}
    &=\ \left(\frac{R}{\eta}\right)^{\sum_{k=1}^dn_k-d}\prod_{k=1}^d\vol\,\bbS^{n_k-2}\int\limits_{\substack{t_1,\dots, t_d\geq 0,\\\sum_{k=1}^d\arcsin(R\eta^{-1}t_k)\leq R}}\,\prod_{k=1}^d\frac{t_k^{n_k-2}}{\sqrt{1-R^2\eta^{-2}t_k^2}}\,\mathrm{d}t_1\cdots\mathrm{d}t_d,
\end{align*}
where we use $\Vector{x}=\yuproj(\Vector{z})$ in the first line, $z_k=\rho_k\theta_k$ with $\rho_k\geq 0$ and $\theta_k\in\bbS^{n_k-2}$ in the second line, $\rho_k=\tan\phi_k$ in the third line, and $\sin\phi_k=R\eta^{-1} t_k$ in the fourth line.

Observe that the domain of integration of the last integral is contained in $[0,\eta]^d$. In this way, we have for each $k$, $\arcsin(R\eta^{-1}t_k)\leq R\eta^{-1}t_k/\sqrt{1-R^2}$, and so the domain of integration contains $\sum_{k=1}^d t_k\leq \eta\sqrt{1-R^2}$. Since $1\big/\sqrt{1-R^2\eta^{-2}t^2_k}\geq 1$, we obtain the following lower bound:
\[
\vol\,B_\bbS\left(\mathbf{e}_1,R\right)\ \geq\ \left(R\sqrt{1-R^2}\right)^{\sum_{k=1}^dn_k-d}\prod_{k=1}^d\vol\,\bbS^{n_k-2}\int\limits_{\substack{t_1,\dots, t_d\geq 0,\\\sum_{k=1}^dt_k\leq 1}}\,\prod_{k=1}^d t_k^{n_k-2}\,\mathrm{d}t_1\cdots\mathrm{d}t_d
\]
where we took $\eta^{-1}=\sqrt{1-R^2}$. Therefore, taking $R=\frac{1}{2L}$, we only have to show that
\begin{equation}\label{eq:boundCprimitive}
    C(L,d;n_1,\ldots,n_d)\ \geq\ \frac{\left(\frac{2L}{\sqrt{1-\frac{1}{4L^2}}}\right)^{\sum_{k=1}^dn_k-d}\prod_{k=1}^d\frac{\vol\,\bbS^{n_k-1}}{\vol\,\bbS^{n_k-2}}}{\int\limits_{\substack{t_1,\dots, t_d\geq 0,\\\sum_{k=1}^dt_k\leq 1}}\,\prod_{k=1}^d t_k^{n_k-2}\,\mathrm{d}t_1\cdots\mathrm{d}t_d}
\end{equation}
for the chosen value of $C(L,d;n_1,\ldots,n_d)$. We prove this bound, by bounding the three parts of the logarithm of the right-hand side separately.

First, we have that $\ln\frac{1}{\sqrt{1-\frac{1}{4L^2}}}\leq \frac{1}{6L^2}$ since $L\geq 1$. Thus
\begin{equation}\label{eq:boundfactor1}
    \ln \left(\frac{2L}{\sqrt{1-\frac{1}{4L^2}}}\right)^{\sum_{k=1}^dn_k-d}\ \leq\  \left(\sum_{k=1}^dn_k-d\right)\left(\ln 2+\frac{1}{6L^2}+\ln L\right).
\end{equation}

Second, for $n_1,\dots, n_d\geq 2$, Lemma $2.25$ from \cite{conditionbook} gives
\[
\frac{\vol\bbS^{n_k-1}}{\vol\,\bbS^{n_k-2}}\ \leq\ \frac{\sqrt{2\pi n_k}}{n_k-1}\ \leq\  \frac{2\sqrt{\pi}}{\sqrt{n_k-1}}
\]
and hence
\begin{equation}\label{eq:boundfactor2}
\ln\prod_{k=1}^d\frac{\vol\,\bbS^{n_k-1}}{\vol\,\bbS^{n_k-2}}\ \leq\ d\ln(2\sqrt{\pi})-\frac{1}{2}\sum_{k=1}^d\ln(n_k-1).
\end{equation}

Third, we have that
\begin{align*}
    \ln&\left(\ \int\limits_{\substack{t_1,\dots, t_d\geq 0,\\\sum_{k=1}^dt_k\leq 1}}\,\prod_{k=1}^d t_k^{n_k-2}\,\mathrm{d}t_1\cdots\mathrm{d}t_d\right)\
    =\ -\ln d! + \ln\left(d!\int\limits_{\substack{t_1,\dots, t_d\geq 0,\\ \sum_{k=1}^dt_k\leq 1}}\,\prod_{k=1}^d t_k^{n_k-2}\,\mathrm{d}t_1\cdots\mathrm{d}t_d\right)& 
      \end{align*}    
\begin{align*}
       &\geq\  -\ln d!+\sum_{k=1}^dd!(n_k-2)\int\limits_{\substack{t_1,\dots, t_d\geq 0,\\\sum_{k=1}^dt_k\leq 1}}\,\ln t_k\,\mathrm{d}t_1\cdots\mathrm{d}t_d&\text{(Jensen's inequality)}\\
    &=\ -\ln d!+\sum_{k=1}^dd(n_k-2)\int_{0}^1\,(1-t_k)^{d-1}\ln t_k\,\mathrm{d}t_k&\text{(Integrate over }\sum_{i\neq k}t_i\leq 1-t_k\text{)}\\
&=\ -\ln d!-\left(d\int_{0}^1\,\sum_{l=1}^{\infty}\frac{(1-t)^{d+l-1}}{l}\,\mathrm{d}t\right)\left(\sum_{k=1}^d n_k-2d\right)&\left(\ln t=-\sum_{l=1}^\infty \frac{(1-t)^l}{l}\right)\\
    &=\ -\ln d!-\left(d\sum_{l=1}^{\infty}\int_{0}^1\,\frac{(1-t)^{d+l-1}}{l}\,\mathrm{d}t\right)\left(\sum_{k=1}^d n_k-2d\right)&\text{(Monotone convergence)}\\
    &=\ -\ln d!-\left(\sum_{l=1}^{\infty}\frac{d}{l(d+l)}\right)\left(\sum_{k=1}^d n_k-2d\right)\\
    &=\ -\ln d!-\left(\sum_{k=1}^d\frac{1}{k}\right)\left(\sum_{k=1}^d n_k-2d\right).&\left(\frac{d}{l(d+l)}=\frac{1}{l}-\frac{1}{d+l}\right).
\end{align*}
Now, Stirling's bound~\cite[Eq.~2.14]{conditionbook} gives
\[
\ln d!\ \leq\ \frac{1}{2}\ln(2\pi)+d\ln d+\frac{1}{2}\ln d-d+\frac{1}{12d}.
\]
Formula $(3)$ in \cite[1.2.7]{knuth_vol1}, whose proof is contained in \cite[1.2.11.2]{knuth_vol1}, yields
\[
\sum_{k=1}^d\frac{1}{k}\ \leq\ \ln d+\gamma+\frac{1}{2d},
\]
where $\gamma:=\lim_{n\to\infty}\left(-\ln n+\sum_{k=1}^n 1/k\right)=0.57721566\ldots$ is the Euler–Mascheroni constant. Thus, we have that
\begin{equation}\label{eq:boundfactor3}
    -\ln\left(\ \int\limits_{\substack{t_1,\dots, t_d\geq 0,\\\sum_{k=1}^dt_k\leq 1}}\,\prod_{k=1}^d t_k^{n_k-2}\,\mathrm{d}t_1\cdots\mathrm{d}t_d\right)\ \leq\ \left(\sum_{k=1}^dn_k\right)\left(\ln d+\gamma+\frac{1}{2d}\right)-d\ln d-d,
\end{equation}
since $\frac{1}{2}\ln(2\pi)-1<0$ and $\frac{1}{2}\ln d+\frac{1}{12d}-2\gamma d<0$.

Finally, applying the logarithm to~\eqref{eq:boundCprimitive} and using the inequalities \eqref{eq:boundfactor1}, \eqref{eq:boundfactor2} and \eqref{eq:boundfactor3}, we obtain the desired bound after noticing that $\gamma+\frac{1}{2d}+\ln 2+\frac{1}{6L^2}\leq 2$.
\end{proof}
\begin{remark}The above theorem can be also proven by means of nets, see \cite[Section~5.1]{vershyninbook}. In both proofs, we wound have to bound the quantity
\[
\frac{\prod_{k=1}^d\vol\bbS^{n_k-1}}{\vol\,B_\bbS\left(\mathbf{e}_1,(2L)^{-1}\right)}.
\]
However, no net will achieve the above bound exactly, since we cannot cover $\prod_{k=1}^d\bbS^{n_k-1}$ with balls of radius $(2L)^{-1}$ without overlapping them. Our argument avoids, on the one hand, dealing with the construction of the net, and, on the other hand, gives better bounds.
\end{remark}

\subsection{Upper bound for general tensors}

We now apply Theorem \ref{theo:generalboundtheo} to give bounds on \eqref{eq:BROA}.

\begin{prop}\label{cor:bound_general}
Let $\mathcal{T}\in \K{K}^\mathbf{n}$ be a Gaussian tensor.
Then
\begin{align}\label{eq:Ebound}
    \K{E}_{\mathcal{T}}\frac{\Vert \mathcal{T}\Vert_{\infty}}{\Vert \mathcal{T}\Vert\ \ }\ \leq\ 9\left(1+\frac{1}{\ln d}+\frac{2}{d+\sum_j n_j}\right) \frac{\sqrt{d\ln d}}{\sqrt{\min_i \prod_{j\neq i} n_j}}.
\end{align}
Moreover,
\begin{align*}
   \frac{1}{\sqrt{\min_{i}\prod_{j\neq i} n_j}}\ \leq\  \mathcal{A}\left(\K{K}^\mathbf{n}\right)\ \leq\ 2\sqrt{3e}\sqrt{1+\frac{2}{\ln d}} \frac{\sqrt{d\ln d}}{\sqrt{\min_i \prod_{j\neq i} n_j}}.
\end{align*}
\end{prop}

\begin{proof}
Let us consider a random Lipschitz function
\begin{equation}\label{eq:Lipschitz}
    \begin{aligned}
\fkF: \K{S}(\K{K}^{n_1})\times\dots\times \K{S}(\K{K}^{n_d})\ &\rightarrow\ [0,\infty),\\
\Vector{x} = (\Vector{x}^1,\dots,\Vector{x}^d)\ &\mapsto\  \frac{\vert \langle \mathcal{T},\Vector{x}^1\otimes\dots\otimes \Vector{x}^d\rangle \vert}{\Vert \mathcal{T}\Vert}.
\end{aligned}
\end{equation}
Its Lipschitz constant satisfies $\mathrm{Lip}(\fkF)\leq \max_{\Vector{x}^j\in \K{S}(\K{K}^{n_j})} \fkF(\Vector{x})=\frac{\Vert \mathcal{T}\Vert_\infty}{\Vert \mathcal{T}\Vert\ \ }$, since 
\begin{equation}\label{eq:Lip_bound}
    \begin{aligned}
\vert \fkF(\Vector{x})-\fkF(\Vector{y})\vert\ &\leq\ \frac{\vert \langle\mathcal{T},\Vector{x}^1\otimes\dots\otimes \Vector{x}^d-\Vector{y}^1\otimes \dots\otimes\Vector{y}^d\rangle \vert}{\Vert \mathcal{T}\Vert}\\
&\leq\ \sum_{j=1}^d \frac{\vert\langle \mathcal{T},\Vector{x}^1\otimes \dots\otimes\Vector{x}^{j-1}\otimes (\Vector{x}^j-\Vector{y}^j)\otimes \Vector{y}^{j+1}\otimes\dots\otimes \Vector{y}^d\rangle\vert}{\Vert\mathcal T\Vert}\\
&\leq\ \left(\max_{\Vector{z}^j\in \K{S}(\K{K}^{n_j})} \fkF(\Vector{z})\right)\sum_{j=1}^d \Vert \Vector{x}^j-\Vector{y}^j\Vert_2\ \leq\ \left(\max_{\Vector{z}^j\in \K{S}(\K{K}^{n_j})} \fkF(\Vector{z})\right)\mathrm{dist}_\K{S}(\Vector{x},\Vector{y})
    \end{aligned}
\end{equation}
holds for any $\Vector{x}, \Vector{y}\in \K{S}(\K{K}^{n_1})\times\dots\times \K{S}(\K{K}^{n_d})$.
By Theorem \ref{theo:generalboundtheo}, we have for all $t>0$,
\begin{equation}\label{eq:tail_general}
    \K{P}_\mathcal{T}\left( \frac{\Vert\mathcal{T}\Vert_\infty}{\Vert \mathcal{T}\Vert\ \ } \geq t \right)\ \leq\ C(d,\mathbf{n})\  \K{P}_\mathcal{T}\left(\fkF(\Vector{e}_1,\dots, \Vector{e}_1)\geq \frac{t}{2}\right)
    ,
\end{equation}
where $\ln C(d,\mathbf{n})= (2+\ln d)\left(\sum_{j=1}^d kn_j\right) - \frac{1}{2}\sum_{j=1}^d \ln(kn_j-1)-d\ln d$ with $k$ being either  $1$ ($\K{K}=\K{R}$) or $2$ ($\K{K}=\K{C}$). Since $\fkF(\Vector{e}_1,\dots, \Vector{e}_1)=\Vert\langle \mathcal{T},\Vector{e}_1\otimes \dots\otimes \Vector{e}_1\rangle\, \Vector{e}_1\otimes \dots\otimes \Vector{e}_1\Vert/\Vert \mathcal{T}\Vert $, Proposition \ref{prop:projection} and Remark \ref{rem:projection} applied to the orthogonal (unitary, if $\K{K}=\K{C}$) projection $T\mapsto \langle T,\Vector{e}_1\otimes \dots\otimes \Vector{e}_1\rangle\, \Vector{e}_1\otimes \dots\otimes \Vector{e}_1$ yield
\begin{align*}
    \K{P}_\mathcal{T}\left(\fkF(\Vector{e}_1,\dots, \Vector{e}_1) \geq \frac{t}{2}\right)\ \leq\ 3\exp\left(-\frac{kn_1\cdots n_d}{3e^{k-1}} \frac{t^2}{4}\right),
\end{align*}
which combined with \eqref{eq:tail_general} gives
\begin{align}\label{eq:tail_general2}
    \K{P}_\mathcal{T}\left( \frac{\Vert\mathcal{T}\Vert_\infty}{\Vert \mathcal{T}\Vert\ \ } \geq t \right)\ \leq\ 3C(d,\mathbf{n}) \exp\left(-\frac{kn_1\cdots n_d}{12e^{k-1}} \,t^2\right).
\end{align}
By Proposition \ref{prop:subgaussiantailbound} we finally have
\begin{equation*}
\begin{aligned}
\K{E}_\mathcal{T} \frac{\Vert \mathcal{T}\Vert_\infty}{\Vert \mathcal{T}\Vert\ \ }
&\leq\ \frac{2\sqrt{6} e^{\frac{k-1}{2}}}{\sqrt{kn_1\cdots n_d}}\sqrt{\ln 3C(d,\mathbf{n})}\left(1+\frac{1}{\ln 3C(d,\mathbf{n})}\right) \\
&\leq\ 2\sqrt{6}e^{\frac{k-1}{2}} \sqrt{\ln d}\frac{\sqrt{\sum_{j=1}^d n_j}}{\sqrt{\prod_{j=1}^d n_j}}\sqrt{1+\frac{2}{\ln d}}\left(1+\frac{1}{d+\sum_j kn_j}\right)\\
 &\leq\ 2\sqrt{6}e^{\frac{k-1}{2}}\left(1+\frac{1}{\ln d}+\frac{2}{d+\sum_j kn_j}\right)\frac{\sqrt{d\ln d}}{\sqrt{\min_i \prod_{j\neq i} n_j}},
\end{aligned}
\end{equation*}
where in the second line we use the estimates of Remark~\ref{rem:boundsC}. The bound for \eqref{eq:BROA} follows by applying the trick in Remark~\ref{remark:mintrick} to \eqref{eq:tail_general2}.
\end{proof}

\begin{remark}
When all dimensions $n_1=\dots=n_d=n$ are equal, one has
\begin{align*}
   \frac{1}{n^{\frac{d-1}{2}}}\ \leq\  \mathcal{A}\left(\K{K}^\mathbf{n}\right)\ \leq\  \K{E}_{\mathcal{T}} \frac{\Vert \mathcal{T}\Vert_\infty}{\Vert \mathcal{T}\Vert\ \ }\ \leq\ 9\left(1+\frac{1}{\ln d}+\frac{2}{d(1+n)}\right)\frac{\sqrt{d\ln d}}{n^{\frac{d-1}{2}}}. 
\end{align*}

\end{remark}

\subsection{Upper bounds for symmetric tensors}

We now apply Theorem \ref{theo:generalboundtheo} to give bounds on \eqref{eq:SBROA}. We formulate everything in the equivalent terms of homogeneous polynomials, see Section \ref{sub:BROA}. First, we consider random Kostlan forms and then random harmonic forms, which deliver better bounds for \eqref{eq:SBROA} in the real case.

\begin{prop}\label{cor:sym}
Let $\fkf\in \PP_{d,n}$ be a Kostlan form. Then
\begin{align}\label{eq:Ebound_sym}
    \K{E}_\fkf \frac{\Vert \fkf\Vert_\infty}{\Vert \fkf\Vert\ \ }\ \leq\ 9 \left(1+\frac{1}{\ln d}+\frac{1}{1+n}\right)\sqrt{n\ln d}\binom{d+n-1}{d}^{-\frac{1}{2}}.
\end{align}
Moreover,
\begin{align*}
     \mathcal{A}(\Sym^d(\K{K}^n))\ \leq\ 2\sqrt{3e}\sqrt{1+\frac{2}{\ln d}} \sqrt{n\ln d}\binom{d+n-1}{d}^{-\frac{1}{2}}. 
\end{align*}
\end{prop}
\begin{prop}\label{cor:harm}
Let $\fkh\in \mathrm{H}_{d,n}$ be a Gaussian harmonic form. Then
\begin{equation}\label{eq:harm_bound}
\K{E}_\fkh\frac{\Vert \fkh\Vert_\infty}{\Vert \fkh\Vert\phantom{_\infty}}\ \leq\ 2\sqrt{6}\left(1+\frac{1}{\ln d}+\frac{1}{n+1}\right)\frac{\sqrt{n\ln d}}{2^{\frac{d}{2}}}\binom{d+\frac{n}{2}-1}{d}^{-\frac{1}{2}}.
\end{equation}
Moreover,
\begin{equation}
\mathcal{A}(\Sym^d(\K{R}^n))\ \leq\ 2\sqrt{3}\sqrt{1+\frac{2}{\ln d}}\,\frac{\sqrt{n\ln d}}{2^{\frac{d}{2}}}\binom{d+\frac{n}{2}-1}{d}^{-\frac{1}{2}}.
\end{equation}
\end{prop}

\begin{proof}[Proof of Proposition~\ref{cor:sym}]
Let us consider a random Lipschitz function
\begin{equation*}
    \begin{aligned}
    \fkF: \K{S}(\K{K}^n)\ &\rightarrow\ [0,\infty),\\
\Vector{x}\ &\mapsto\  \frac{\vert \fkf(\Vector{x}) \vert}{\Vert \fkf \Vert}.
    \end{aligned}
\end{equation*}
Note that it is the restriction of the function \eqref{eq:Lipschitz} to the diagonally embedded sphere $\K{S}(\K{K}^n)\hookrightarrow \K{S}(\K{K}^n)\times \dots\times \K{S}(\K{K}^n)$.
By \eqref{eq:Lip_bound} and \eqref{eq:Banach} the Lipschitz constant of $\fkF$ satisfies \begin{align}\label{eq:Lip_sym}
    \mathrm{Lip}(\fkF)\ \leq\ d \max_{\Vector{x}\in \K{S}(\K{K}^n)} \fkF(\Vector{x})\ =\ d\frac{\Vert \fkf\Vert_\infty}{\Vert \fkf\Vert\ \ }.
    \end{align}
By Theorem \ref{theo:generalboundtheo} we have for all $t>0$,

\begin{align}\label{eq:tail_symmetric}
    \K{P}_\fkf\left(\frac{\,\Vert \fkf\Vert_\infty}{\Vert \fkf\Vert\ \ }\geq t\right)\ \leq\ C(d,n)\, \K{P}_\fkf\left( \fkF(\Vector{e}_1)\geq \frac{t}{2} \right)\ =\ C(d,n)\, \K{P}_\fkf\left( \frac{\vert \fkf_{(d,0,\dots,0)}\vert}{\Vert \fkf\Vert }\geq \frac{t}{2} \right),
\end{align}
where $\ln C(d,n)= (2+\ln d) kn-\frac{1}{2} \ln(kn-1)-\ln d$ with $k$ being either $1$ ($\K{K}=\K{R}$) or 2 ($\K{K}=\K{C}$).
Since $\fkf$ is a Kostlan form, that is, $\fkf_\alpha=\sqrt{\binom{d}{\alpha}}\tilde \fkf_\alpha$, where variables $\tilde \fkf_\alpha$ are independent standard (complex) Gaussians, its Bombieri-Weyl norm satisfies
\begin{align*}
    \Vert \fkf\Vert^2\ =\ \sum_{\vert\alpha\vert=d} \binom{d}{\alpha}^{-1} \vert \fkf_\alpha\vert^2\ =\ \sum_{\vert\alpha\vert=d} \vert \tilde \fkf_\alpha\vert^2.
\end{align*}
We apply Proposition \ref{prop:projection} and Remark \ref{rem:projection} to the projection $\pi_1$ onto the first coordinate axis in $\mathbb{K}^N$ and, using $\Vert \pi_1(\fkf)\Vert_2 = \vert \tilde \fkf_{(d,0,\dots,0)}\vert=\vert \fkf_{(d,0,\dots,0)}\vert$, obtain that
\begin{align*}
    \K{P}_\fkf \left(\frac{\vert \fkf_{(d,0,\dots,0)}\vert}{\Vert \fkf\Vert}\geq \frac{t}{2}\right)\ \leq\ 3\exp\left(-\frac{kN}{3e^{k-1}} \frac{t^2}{4}\right),
\end{align*}
which combined with \eqref{eq:tail_symmetric} gives
\begin{align}\label{eq:tail_symmetric2}
     \K{P}_\fkf\left(\frac{\,\Vert \fkf\Vert_\infty}{\Vert \fkf\Vert\ \ }\geq t\right)\ \leq\ 3C(d,n) \exp\left(-\frac{kN}{12e^{k-1}}\, t^2\right).
\end{align}

This, together with inequality \eqref{eq:tail_symmetric} and Proposition \ref{prop:subgaussiantailbound}, finally implies that

\begin{equation*}
\begin{aligned}
\K{E}_\fkf\frac{\Vert \fkf\Vert_\infty}{\Vert \fkf\Vert\ \ \,}\ &\leq\ \frac{2\sqrt{6} e^{\frac{k-1}{2}}}{\sqrt{kN}}\sqrt{\ln 3C(d,n)}\left(1+\frac{1}{\ln 3C(d,n)}\right)\\
&\leq\ 2\sqrt{6} e^{\frac{k-1}{2}}\left(1+\frac{1}{\ln d}\right)\left(1+\frac{1}{2(1+n)}\right)\sqrt{n\ln d}\binom{n+d-1}{d}^{-\frac{1}{2}}\\
&\leq\  2\sqrt{6} e^{\frac{k-1}{2}}\left(1+\frac{1}{\ln d}+\frac{1}{1+n}\right)\sqrt{n\ln d}\binom{n+d-1}{d}^{-\frac{1}{2}}\\
&\leq\ 2\sqrt{6} e^{\frac{k-1}{2}}\left(1+\frac{1}{\ln d}+\frac{1}{1+n}\right)
\frac{\sqrt{d!\ln d}}{n^{\frac{d-1}{2}}},
\end{aligned}
\end{equation*}
where in the second line we use bounds from Remark~\ref{rem:boundsC} with $d\geq 3$. 
The bound for \eqref{eq:SBROA} follows by applying the trick in Remark~\ref{remark:mintrick} to \eqref{eq:tail_symmetric2}.
\end{proof}

\begin{proof}[Proof of Proposition~\ref{cor:harm}]
Given a Gaussian harmonic form $\fkh\in \mathrm{H}_{d,n}$, we consider a random Lipschitz function
\begin{equation*}
\begin{aligned}
    \fkF: \K{S}^{n-1}&\rightarrow [0,\infty)],\\
    \Vector{x} & \mapsto \frac{\vert \fkh(\Vector{x})\vert}{\Vert \fkh\Vert}.
\end{aligned}
\end{equation*}
The bound \eqref{eq:Lip_sym} implies
that the Lipschitz constant of $\fkF$ satisfies
\begin{align*}
    \textrm{Lip}(\fkF)\ \leq\ d\max_{\Vector{x}\in \K{S}^{n-1}}\fkF(\Vector{x})\ =\ d\frac{\Vert \fkh\Vert_{\infty}}{\Vert \fkh\Vert\phantom{_\infty}}.
\end{align*}
By Theorem \ref{theo:generalboundtheo} we have for all $t>0$ that
\begin{align}\label{eq:first_harm}
    \mathbb{P}_{\fkh}\left(\frac{\Vert \fkh\Vert_\infty}{\Vert \fkh\Vert\phantom{_\infty}}\geq t\right)\ \leq\ C(d,n) \max_{\Vector{x}\in \K{S}^{n-1}} \mathbb{P}_\fkh \left(\fkF(\Vector{x})\geq \frac{t}{2}\right)
\end{align}
with $\ln C(d,n)=(2+\ln d)n-\frac{1}{2}\ln(n-1)-\ln d$.
Since the inner product \eqref{eq:L^2} of two harmonic forms is invariant under orthogonal changes of variables and since, by Lemma \ref{lem:comparison}, it is proportional to \eqref{eq:BW}, the random variables $\fkF(\Vector{x})$  and $\fkF(\Vector{x}')$ have the same distribution  for any $\Vector{x}, \Vector{x}'\in \K{S}^{n-1}$. 
In particular, in the right-hand side of \eqref{eq:first_harm} we can drop the maximum and consider any point $\Vector{x}\in \K{S}^{n-1}$.

The evaluation of $h\in \mathrm{H}_{d,n}$ at a point $\Vector{x}\in \mathbb{S}^{n-1}$ does not anymore correspond to an orthogonal projection in $(\mathrm{H}_{d,n},\langle\cdot,\cdot\rangle_{\LL})$, as it was in the case of Kostlan polynomials. However, the formula \eqref{eq:zonal} implies that it is given by taking inner product with $Z_{\Vector{x}}$. So, 
\begin{equation*}
\begin{aligned}
P: \mathrm{H}_{d,n} &\rightarrow \K{R}Z_{\Vector{x}},\\
h & \mapsto \frac{h(\Vector{x})}{\Vert Z_{\Vector{x}}\Vert_{\LL}} \frac{Z_{\Vector{x}}}{\Vert Z_{\Vector{x}}\Vert_{\LL}},
\end{aligned}
\end{equation*}
is an orthogonal projection on the line through $Z_{\Vector{x}}$ and Proposition \ref{prop:projection} gives
\begin{align*}
    \K{P}_\fkh\left(\frac{\vert \fkh(\Vector{x})\vert}{\Vert Z_{\Vector{x}}\Vert_{\LL} \Vert \fkh\Vert_{\LL}}\geq t\right)\ \leq\ 3 \exp\left(-\frac{D_{d,n}}{3} t^2\right),\quad t\geq 0.
\end{align*}
This and Lemma \ref{lem:comparison} imply that for all $t\geq 0$,
\begin{equation*}
\begin{aligned}
    \K{P}_\fkh\left(\fkF(\Vector{x})\geq \frac{t}{2}\right)\ &=\ \K{P}_\fkh\left(\frac{\vert \fkh(\Vector{x})\vert}{\Vert Z_{\Vector{x}} \Vert_{\LL} \Vert \fkh\Vert_{\LL}}\geq \frac{t \sqrt{\Gamma\left( d+\frac{n}{2}\right)}2^{\frac{d-1}{2}}}{2\Vert Z_{\Vector{x}}\Vert_{\LL} \sqrt{\Gamma(d+1)} \pi^{\frac{n}{4}}}\right)\\
    &\leq\ 3\exp\left(  -\,\frac{D_{d,n}\Gamma\left( d+\frac{n}{2}\right) 2^{\,d-1}}{\Vert Z_{\Vector{x}}\Vert^2_{\LL} \Gamma(d+1) \pi^{\frac{n}{2}} } \,\frac{t^2}{12}\right).
\end{aligned}
\end{equation*}
Now, combining this with \eqref{eq:first_harm}, applying Proposition \ref{prop:subgaussiantailbound} and proceeding as in the proof of Proposition \ref{cor:sym}, we derive
\begin{equation}\label{eq:bound_Eh}
    \begin{aligned}
    \K{E}_\fkh\frac{\Vert \fkh\Vert_{\infty}}{\Vert \fkh\Vert\phantom{_\infty}}\ &\leq\ \sqrt{2}K\sqrt{\ln 3C(d,n)}\left(1+\frac{1}{\ln 3C(d,n)}\right)\\
    &\leq\ \sqrt{2}\left(1+\frac{1}{\ln d}+\frac{1}{n+1}\right)\sqrt{n\ln d}\,K,
\end{aligned} \end{equation}
where
    \begin{align*}
        K\ =\ \frac{\sqrt{3}\Vert Z_{\Vector{x}}\Vert_{\LL} \sqrt{\Gamma(d+1)} \pi^{\frac{n}{4}}}{\sqrt{D_{d,n}} \sqrt{\Gamma\left( d+\frac{n}{2}\right)}2^{\frac{d-3}{2}}}.
    \end{align*}
    By \eqref{eq:normZ} and the formula $\vert \K{S}^{n-1}\vert = 2\pi^{\frac{n}{2}}/\Gamma\left(\frac{n}{2}\right)$ for the volume of the sphere, we  simplify:
    \begin{align*}
        K\ =\  \sqrt{\frac{12d!\Gamma\left(\frac{n}{2}\right)}{2^{d}\Gamma\left( d+\frac{n}{2}\right)}}\ =\ \frac{2\sqrt{3}}{2^{\frac{d}{2}}}\binom{d+\frac{n}{2}-1}{d}^{-\frac{1}{2}}.
    \end{align*}
    We combine this expression for $K$ with \eqref{eq:bound_Eh} and finally obtain
    \begin{equation*}
        \begin{aligned}
          \K{E}_\fkh\frac{\Vert \fkh\Vert_{\infty}}{\Vert \fkh\Vert\phantom{_\infty}}\ &\leq\  
          2\sqrt{6}\left(1+\frac{1}{\ln d}+\frac{1}{n+1}\right)\frac{\sqrt{n\ln d}}{2^{\frac{d}{2}}}\binom{d+\frac{n}{2}-1}{d}^{-\frac{1}{2}}       \\  
          &\leq\ 2\sqrt{6}\left(1+\frac{1}{\ln d}+\frac{1}{n+1}\right)\sqrt{d!\ln d}\frac{1}{n^{\frac{d-1}{2}}}, 
\end{aligned}
\end{equation*}
as desired. The bound for \eqref{eq:SBROA} is obtained by applying the trick in Remark~\ref{remark:mintrick}. 
\end{proof}   

\subsection{Upper bounds for partially symmetric tensors}

The following propositions provide upper bounds for \eqref{eq:PSBROA}. We state them in the equivalent terms of multi-ho\-mo\-ge\-neous polynomials. As our proof strategies are similar to those in Propositions \ref{cor:sym} and \ref{cor:harm} respectively, we leave out some details.

\begin{prop}\label{cor:part_symA}
Let $m\geq 1$, $d_1,\ldots,d_m\geq 2$ with $\max_j d_j\geq 3$ and $n_1,\ldots,n_m\geq 2$. Then the quantity $\mathcal{A}\left(\bigotimes_{j=1}^m\Sym^{d_j}(\K{K}^{n_j})\right)$ is upper bounded by
\begin{align*}
   2\sqrt{3e}\sqrt{1+\frac{2}{\ln \left(m\max_{j} d_j\right)}}\sqrt{\left(\sum_{j=1}^m n_j\right)\ln \left(m\max_j d_j\right)}\prod_{j=1}^m\binom{d_j+n_j-1}{d_j}^{-\frac{1}{2}}.
\end{align*}
\end{prop}
\begin{prop}\label{cor:part_symB}
Let $m\geq 1$, $d_1,\ldots,d_m\geq 2$ with $\max_j d_j\geq 3$ and $n_1,\ldots,n_m\geq 2$. Then the quantity $\mathcal{A}\left(\bigotimes_{j=1}^m\Sym^{d_j}(\K{R}^{n_j})\right)$ is upper bounded by
\begin{align*}
   2\sqrt{3}\sqrt{1+\frac{2}{\ln \left(m\max_{j} d_j\right)}}\frac{\sqrt{\left(\sum_{j=1}^m n_j\right)\ln \left(m\max_{j} d_j\right)}}{2^{\frac{1}{2}\sum_{j=1}^m d_j}}\prod_{j=1}^m\binom{d_j+\frac{n_j}{2}-1}{d_j}^{-\frac{1}{2}}.
\end{align*}
\end{prop}
\begin{proof}[Proof of Proposition~\ref{cor:part_symA}]
Let $\mathcal{F}\in \PP_{\mathbf{d},\mathbf{n}}$ be a Kostlan multi-homogeneous polynomial, where  $\mathbf{d}=(d_1,\dots, d_m)$ and $\mathbf{n}=(n_1,\dots, n_m)$. Then, by \eqref{eq:Lip_bound}, the Lipschitz constant of
\begin{align*}
    \mathfrak{F}: \K{S}(\K{K}^{n_1})\times\dots\times \K{S}(\K{K}^{n_m}) &\rightarrow\ \K{R},\\
   \Vector{x}\ =\ (\Vector{x}^1,\ldots,\Vector{x}^m)\ &\mapsto\ \frac{|\mathcal{F}(\Vector{x}^1,\ldots,\Vector{x}^m)|}{\Vert\mathcal{F}\Vert},
\end{align*}
is bounded from above by  $\left(\max_j d_j\right) \max_{\Vector{x}\in \prod_{j=1}^m \K{S}(\K{K}^{n_j})} \vert \mathfrak{F}(\Vector{x})\vert$. Applying Theorem~\ref{theo:generalboundtheo} and using the fact that the Kostlan distribution on $\PP_{\mathbf{d},\mathbf{n}}$ is invariant under the action of the product of orthogonal (respectively, unitary) groups, we obtain
\begin{align}
    \mathbb{P}_{\mathfrak{F}} \left( \max_{\Vector{x}\in \prod_{j=1}^m \K{S}(\K{K}^{n_j})}\mathfrak{F}(\Vector{x})\geq t\right)\ \leq\ C(\mathbf{d},\mathbf{n})\,\K{P}_{\mathfrak{F}}\left(\frac{\vert \mathcal{F}(\Vector{e}^1,\dots, \Vector{e}^1)\vert}{\Vert\mathcal{F}\Vert} \geq \frac{t}{2}\right), 
\end{align}
where 
\begin{align}\label{eq:ln(C)}
    \ln C(\mathbf{d},\mathbf{n})\ =\ (2+\ln(m\max_j d_j))(\sum_{j=1}^m kn_j) -\frac{1}{2}\sum_{j=1}^m \ln(kn_j-1)-m\ln(m\max_j d_j).
\end{align}
Since $F\mapsto F(\Vector{e}^1,\dots, \Vector{e}^1)(x^1_1)^{d_1}\cdots (x^m_1)^{d_m}$ is an orthogonal (respectively, unitary) projection, Proposition \ref{prop:projection} applied to the last term in the above inequality yields
\begin{align}
    \mathbb{P}_{\mathcal{F}} \left(\frac{\ \ \Vert \mathcal{F}\Vert_\infty}{\Vert \mathcal{F}\Vert} \geq t\right)\ \leq\ 3C(\mathbf{d},\mathbf{n})\,
\exp\left(-\frac{N}{12 e^{k-1}}\, t^2\right)
\end{align}
with $N=\prod_{j=1}^m k{d_j+n_j-1\choose d_j}$ being the dimension of $\PP_{\mathbf{d},\mathbf{n}}$. 
Finally, applying \eqref{eq:trick} to the obtained tail estimate for the norm ratio yields the claim. \end{proof}

\begin{proof}[Proof of Proposition~\ref{cor:part_symB}]

Consider the subspace $\textrm{H}_{\mathbf{d},\mathbf{n}}:=\otimes_{j=1}^m \textrm{H}_{d_j,n_j}\subseteq \PP_{\mathbf{d},\mathbf{n}}$ of multi-homogeneous harmonic forms. We endow it with the $L^2(\K{S}^{\mathbf{n}-1})$-product defined for (decomposable) elements $H(\Vector{x}^1,\dots, \Vector{x}^m)=\prod_{j=1}^m h_j(\Vector{x}^j)$, $H'(\Vector{x}^1,\dots, \Vector{x}^m)=\prod_{j=1}^m h'_j(\Vector{x}^j)$ via
\begin{align}\label{eq:L^2_multi}
    \langle H, H'\rangle_{L^2(\K{S}^{\mathbf{n}-1})}\ =\ \prod_{j=1}^m \langle h_j, h_j'\rangle_{L^2(\K{S}^{n_j-1})}
\end{align}
and then extended to the whole space $\textrm{H}_{\mathbf{d},\mathbf{n}}$ by multi-linearity. 
Lemma \ref{lem:comparison} directly implies that the Bombieri-Weyl product \eqref{eq:product_multihom} and the $L^2(\K{S}^{\mathbf{n}-1})$-product are proportional,
\begin{align}
    \langle H, H'\rangle\ =\ \prod_{j=1}^m \frac{2^{\,d_j-1}}{\sqrt{\pi}^{\,n_j}}\frac{\Gamma\left(d_j+\frac{n_j}{2}\right)}{\Gamma\left(d_j+1\right)}\, \langle H, H'\rangle_{L^2(\K{S}^{\mathbf{n}-1})}, \quad H, H'\in \textrm{H}_{\mathbf{d},\mathbf{n}}.
\end{align}

By \eqref{eq:zonal} and \eqref{eq:L^2_multi}, the evaluation of $H\in \textrm{H}_{\mathbf{d},\mathbf{n}}$ at any point $\Vector{x}=(\Vector{x}^1,\dots, \Vector{x}^m)\in \K{S}^{\mathbf{n}-1}$ is given by taking inner product with $\textrm{Z}_{\Vector{x}}\in \textrm{H}_{\mathbf{d},\mathbf{n}}$ defined by  $\textrm{Z}_{\Vector{x}}(\Vector{y}):=\prod_{j=1}^m Z_{\Vector{x}^j}(\Vector{y}^j)$.
So, 
\begin{align}
    \textrm{P}: \textrm{H}_{\mathbf{d},\mathbf{n}}\ &\rightarrow\ \K{R}\textrm{Z}_{\Vector{x}},\\
    H\ &\mapsto\ \frac{H(\Vector{x})}{\Vert \textrm{Z}_{\Vector{x}}\Vert_{L^2(\K{S}^{\mathbf{n}-1})}}\frac{\textrm{Z}_{\Vector{x}}}{\Vert \textrm{Z}_{\Vector{x}}\Vert_{L^2(\K{S}^{\mathbf{n}-1})}},
\end{align}
is an orthogonal projection on the line through $\textrm{Z}_{\Vector{x}}$.
Let $\mathfrak{H}\in \textrm{H}_{\mathbf{d},\mathbf{n}}$ be a Gaussian multi-homogeneous harmonic.
Proceeding as in the proof of Proposition \ref{cor:harm}, we obtain that
\begin{align}
    \K{P}_{\mathfrak{H}}\left( \frac{\ \ \Vert \mathfrak{H}\Vert_{\infty}}{\Vert \mathfrak{H}\Vert} \geq t \right)\ \leq\ 3C(\mathbf{d},\mathbf{n})\,\exp\left( - \prod_{j=1}^m \frac{D_{d_j,n_j}\Gamma\left(d_j+\frac{n_j}{2}\right) 2^{\,d_j-1}}{\Vert \textrm{Z}_{\Vector{x}^j}\Vert^2_{L^2(\K{S}^{n_j-1})} \Gamma\left(d_j+1\right) \pi^{\frac{n_j}{2}}}\ \frac{t^3}{12}\right),
\end{align}
where $\ln C(\mathbf{d},\mathbf{n})$ is given by \eqref{eq:ln(C)} with $k=1$. This estimate combined with \eqref{eq:trick} yields the claimed bound after some elementary simplifications.
\end{proof}

\section{Lower bounds for (partially) symmetric tensors}\label{sec:lower}

In this section we prove lower bounds stated in Theorem~\ref{thm:combi} and in  Theorem~\ref{them:combi2}.  
Our new lower bounds rely on the integral representation of the Bombieri-Weyl norm. 

\begin{prop}\label{prop:lowerboundcomplex}
For any $d\geq 1$ and $n\geq 2$
\begin{align}\label{eq:new_complex}
    \mathcal{A}(\Sym^d(\K{C}^n))\ \geq\ \max\left\{\binom{d+n-1}{d}^{-\frac{1}{2}},\ \frac{1}{n^{\frac{d-1}{2}}}\right\}.
\end{align}
\end{prop}
\begin{prop}\label{prop:lowerboundreal}
For any $d\geq 1$ and $n\geq 2$,
\begin{align}\label{eq:new_real}
    \mathcal{A}(\Sym^d(\K{R}^n))\ \geq\ \max\left\{\frac{1}{2^{\frac{d}{2}}}\binom{d+n-1}{d}^{-\frac{1}{2}},\ \frac{1}{n^{\frac{d-1}{2}}}\right\}.
\end{align}
\end{prop}

\begin{proof}[Proof of Proposition~\ref{prop:lowerboundcomplex}]
The inner product \eqref{eq:BW} admits an integral representation
\begin{align*}
    \langle f, f'\rangle\ =\ \binom{d+n-1}{d}\bbE_{\fkz\in\bbS(\K{C}^n)} \overline{f(\Vector{\fkz})}f'(\Vector{\fkz}),
\end{align*}
where $\fkz$ is a random vector uniformly distributed in the sphere $\bbS(\K{C}^n)$. It implies that
\begin{align}\label{eq:integral_norm}
    \Vert f\Vert^2\ =\ \binom{d+n-1}{d} \bbE_{\fkz\in\bbS(\K{C}^n)} \vert f(\Vector{z})\vert^2\ \leq\ \binom{d+n-1}{d} \Vert f\Vert_\infty^2.
\end{align}
and hence the lower bound follows.
\end{proof}

\begin{proof}[Proof of Proposition~\ref{prop:lowerboundreal}]
A result from \cite{Siciak} (see also \cite[($17.5$)]{KL2018}) asserts that the complex and the real uniform norms of  $f\in \PP_{d,n}$ are linked by
\begin{align}\label{eq:comparison}
\Vert f\Vert_{\infty,\K{C}}\ =\ \max_{\Vector{z}\in \bbS(\K{C}^n)}\vert f(\Vector{z})\vert\ \leq\ \sqrt{2}^d\max_{\Vector{x}\in \bbS(\K{R}^n)}\vert f(\Vector{x})\vert\ =\ \sqrt{2}^d\Vert f\Vert_{\infty,\K{R}},
\end{align}
which combined with \eqref{eq:integral_norm} yields the desired bound. 
\end{proof}

We can also provide similar lower bounds for the case of partially symmetric tensors.

\begin{prop}\label{prop:lowerboundcomplexpart}
For any $m\geq 1$, $d_1,\ldots,d_m\geq 1$ and $n_1,\ldots,n_m\geq 2$
\begin{align}\label{eq:new_complexpart}
    \mathcal{A}\left(\bigotimes_{j=1}^m \Sym^{d_j}(\K{C}^{n_j})\right)\ \geq\ \max\left\{\prod_{j=1}^m\binom{d_j+n_j-1}{d_j}^{-\frac{1}{2}},\ \sqrt{\frac{\max_j n_j}{\prod_{j=1}^m n_j^{d_j}}}\right\}.
\end{align}
\end{prop}
\begin{prop}\label{prop:lowerboundrealpart}
For any $m\geq 1$, $d_1,\ldots,d_m\geq 1$ and $n_1,\ldots,n_m\geq 2$
\begin{align}\label{eq:new_realpart}
    \mathcal{A}\left(\bigotimes_{j=1}^m \Sym^{d_j}(\K{R}^{n_j})\right)\ \geq\ \max\left\{\frac{1}{2^{\frac{1}{2}\sum_{j=1}^m d_j}}\prod_{j=1}^m\binom{d_j+n_j-1}{d_j}^{-\frac{1}{2}},\ \sqrt{\frac{\max_j n_j}{\prod_{j=1}^m n_j^{d_j}}}\right\}.
\end{align}
\end{prop}
\begin{proof}[Proof of Proposition~\ref{prop:lowerboundcomplexpart}]
The Bombieri-Weyl norm for multi-homogeneous polynomials that is given by \eqref{eq:product_multihom} admits an integral representation
\begin{equation*}
    \langle F, F'\rangle\ =\ \prod_{j=1}^m\binom{d_j+n_j-1}{d_j}\bbE_{\Vector{\fkz}^1\in\bbS(\K{C}^{n_1})}\cdots \bbE_{\Vector{\fkz}^m\in\bbS(\K{C}^{n_m})} \overline{F(\Vector{\fkz}^1,\ldots,\Vector{\fkz}^m)}F'(\Vector{\fkz}^1,\ldots,\Vector{\fkz}^m),
\end{equation*}
where vectors $\Vector{\fkz}^j\in\bbS(\K{C}^{n_j})$ are indedpendent and uniformly distributed. Hence
{\small\begin{align*}\label{eq:Integral_norm} 
\|F\|^2\ =\ \prod_{j=1}^m\binom{d_j+n_j-1}{d_j}  \bbE_{\Vector{\fkz}^1\in\bbS(\K{C}^{n_1})}\cdots \bbE_{\Vector{\fkz}^m\in\bbS(\K{C}^{n_m})} \vert F(\Vector{\fkz}^1,\dots,\Vector{\fkz}^m)\vert^2\ \leq\ \prod_{j=1}^m\binom{d_j+n_j-1}{d_j}\|F\|^2_\infty.
\end{align*}}
The other lower bound is essentially \eqref{eq:trivial_bound}, since partially symmetric tensors are general tensors of format $(n_1,\ldots,n_1,\ldots,n_m,\ldots,n_m)$.
\end{proof}

\begin{proof}[Proof of Proposition~\ref{prop:lowerboundrealpart}]
Subsequent applications of \eqref{eq:comparison} give
\begin{align*}
    \Vert F\Vert_{\infty,\K{C}}\ &=\ \max_{\Vector{z}^j\in \K{S}(\K{C}^{n_j})} \vert F(\Vector{z}^1,\dots, \Vector{z}^m)\vert\\ 
    &=\ \max_{\Vector{z}^1\in \K{S}(\K{C}^{n_1})}\dots\max_{\Vector{z}^{m-1}\in \K{S}(\K{C}^{n_{m-1}})}\max_{\Vector{z}^m\in \K{S}(\K{C}^{n_m})} \vert F(\Vector{z}^1,\dots, \Vector{z}^{m-1},\Vector{z}^m)\vert\\ 
    &\leq\ \sqrt{2}^{d_m} \max_{\Vector{z}^1\in \K{S}(\K{C}^{n_1})}\dots\max_{\Vector{z}^{m-1}\in \K{S}(\K{C}^{n_{m-1}})} \max_{\Vector{x}^m\in \K{S}(\K{R}^{n_m})} \vert F(\Vector{z}^1,\dots, \Vector{z}^{m-1}, \Vector{x}^m)\vert\\  &\leq\ \dots\ \leq\ \sqrt{2}^{\,\sum_{j=1}^m d_j} \max_{\Vector{x}^j\in \K{S}(\K{R}^{n_j})} \vert F(\Vector{x}^1,\dots, \Vector{x}^{m-1},\Vector{x}^m)\vert\ =\  \sqrt{2}^{\,\sum_{j=1}^m d_j}  \Vert F\Vert_{\infty,\K{R}},
\end{align*}
which combined with the proof of Proposition \ref{prop:lowerboundcomplexpart} implies the claim.
\end{proof}

\section{Estimates for large \texorpdfstring{$d$}{d}}\label{sec:large-d}

In this section we prove the estimates of Theorem~\ref{thm:combi} when $d$ is large.

\begin{prop}\label{prop:binominteger}
Let $d,n\geq 2$ be integers. If $d\geq  n^2/4$, then
\begin{equation}
    \sqrt{\frac{(n-1)!}{d^{n-1}}}\left(1-\frac{n^2}{4d}\right)\ \leq\ \binom{d+n-1}{d}^{-\frac{1}{2}}\ \leq\ \sqrt{\frac{(n-1)!}{d^{n-1}}}.
\end{equation}
Moreover, the upper bound holds for arbitrary $d$ and $n$.
\end{prop}
\begin{prop}\label{prop:binomhalfinteger}
Let $d,n\geq 2$ be integers. If $d\geq  n^2/16$, then
\begin{equation}
    \sqrt{\frac{\Gamma\left(\frac{n}{2}\right)}{d^{\frac{n}{2}-1}}} \left(1-\frac{n^2}{16d}\right)\ \leq\ \binom{d+\frac{n}{2}-1}{d}^{-\frac{1}{2}}\ \leq\ \sqrt{\frac{\Gamma\left(\frac{n}{2}\right)}{d^{\frac{n}{2}-1}}}\left(1+\frac{1}{4d}\right).
\end{equation}
Moreover, the upper bound holds for arbitrary $d$ and $n$.
\end{prop}
\begin{remark}
For arbitrary $d$ and $n$, we can easily see that
\[
\binom{d+n-1}{d}^{-\frac{1}{2}}\ \leq\ \frac{\sqrt{d!}}{n^{\frac{d}{2}}}\quad ~\text{and }\quad ~\binom{d+\frac{n}{2}-1}{d}^{-\frac{1}{2}}\ \leq\ \frac{2^{\frac{d}{2}}\sqrt{d!}}{n^{\frac{d}{2}}}.
\]
Thus, the case when $n$ is large is covered.
\end{remark}
\begin{proof}[Proof of Proposition~\ref{prop:binominteger}]
We have that
\[\binom{d+n-1}{d}\ =\ \frac{1}{(n-1)!}\prod_{k=1}^{n-1}(d+k)\ \geq\ \frac{d^{n-1}}{(n-1)!}.\]

On the other hand, by the AM-GM inquality,
\begin{align*}
    \prod_{k=1}^{n-1}(d+k)&\ \leq\ \left(d+\frac{1}{n-1}\binom{n}{2}\right)^{n-1}&\text{(AM-GM inquality)}\\
    &=\ d^{n-1}\left(1+\frac{1}{d(n-1)}\binom{n}{2}\right)^{n-1}
    \end{align*}\begin{align*}
    &\leq\ d^{n-1}e^{\frac{1}{d}\binom{n}{2}}&\left((1+\frac{x}{N})^N\ \leq\ e^x\right)\\
    &\leq\ \frac{d^{n-1}}{\left(1-\frac{1}{2d}\binom{n}{2}\right)^2}&\left(e^x\ \leq\ \frac{1}{(1-\frac{x}{2})^2}\text{ for }x\in [0,2)\right)\\
    &\leq\ \frac{d^{n-1}}{\left(1-\frac{n^2}{4d}\right)^2}.
\end{align*}
Hence
\[
\binom{d+n-1}{d}\ \leq\ \frac{d^{n-1}}{(n-1)!}\frac{1}{(1-\frac{n^2}{4d})^2}.
\]
Now, the estimates follow.
\end{proof}
\begin{proof}[Proof of Proposition~\ref{prop:binomhalfinteger}]
If $n$ is even, then the estimates follow from the previous proposition. So, we can assume that $n$ is odd. Then
\[
\binom{d+\frac{n}{2}-1}{d}\ =\ \frac{1}{\Gamma\left(\frac{n}{2}\right)}\frac{\Gamma\left(d+\frac{1}{2}\right)}{\Gamma(d+1)}\prod_{k=0}^{\frac{n-1}{2}-1}\left(d+k+\frac{1}{2}\right).
\]
Arguing as in the previous proposition, we have
\[
\frac{d^{\frac{n-1}{2}}}{\Gamma\left(\frac{n}{2}\right)}\frac{\Gamma\left(d+\frac{1}{2}\right)}{\Gamma(d+1)}\ 
\leq\ \binom{d+\frac{n}{2}-1}{d}\ \leq\ \frac{d^{\frac{n-1}{2}}}{\Gamma\left(\frac{n}{2}\right)}\frac{\Gamma\left(d+\frac{1}{2}\right)}{\Gamma(d+1)}\frac{1}{\left(1-\frac{1}{d}\left(\frac{n}{4}\right)^2\right)^2}.
\]
Now, the claimed estimated follow by Gautschi's inequality~\cite[(7)]{gautschi1959},
\[\frac{1}{d^{\frac{1}{2}}}\frac{1}{1+\frac{1}{2d}}\ \leq\ \frac{1}{(1+d)^{\frac{1}{2}}}\ \leq\ \frac{\Gamma\left(d+\frac{1}{2}\right)}{\Gamma(d+1)}\ \leq\ \frac{1}{d^{\frac{1}{2}}}.\]
\end{proof}


\bibliographystyle{plain}
{\small 
\bibliography{biblio}
}

\appendix
\section{Appendix}

For the sake of completeness we include a proof of Proposition~\ref{prop:subgaussiantailbound} and of the formula \eqref{eq:yuprojjacobian}. The latter is needed in the proof of Theorem~\ref{theo:generalboundtheo}.

\subsection{Proof of Proposition~\ref{prop:subgaussiantailbound}}

The proof follows the lines of~\cite[Proposition 2.5.2]{vershyninbook}. We give a more detailed proof than the one given in~\cite[Proposition~4.24]{CETC-norms}. 

We begin with the first statement. Fix $\lambda>0$. Then Markov's inequality implies that
\[
\bbP(|\fkt|\geq t)\ =\ \bbP\left(e^{\lambda^2\fkt^2}\geq e^{\lambda^2t^2}\right)\ \leq\  e^{-\lambda^2t^2}\bbE e^{\lambda^2\fkt^2}.
\]
Now, the Taylor expansion of the exponential function, standard facts from calculus and our assumptions yield
\[
\bbE e^{\lambda^2\fkt^2}\ =\ \sum_{p=0}^\infty \frac{\lambda^{2p}\bbE \fkt^{\,2p}}{p!}\ \leq\ \sum_{p=0}^\infty \frac{(\lambda^22pK^2)^p}{p!}.
\]
With $\lambda=\frac{1}{\sqrt{6}K}$ we obtain that
\[
\bbP(|\fkt|\geq t)\ \leq\ e^{-\frac{t^2}{6K^2}}\sum_{p=0}^\infty \frac{(p/3)^p}{p!}.
\]
Now, by direct computation, $\sum_{p=0}^\infty \frac{(p/3)^p}{p!}\simeq 2.62509\ldots$,
which proves the first assertion.

For the second statement, by \cite[Lemma $1.2.1$]{vershyninbook} we write
\[
\bbE |\fkt|^{\ell}\ =\ \int_{0}^\infty \bbP(\vert \fkt\vert^{\ell}\geq u)\,\mathrm{d}u\ =\ \int_0^\infty\,\ell t^{\ell-1}\bbP(|\fkt|\geq t)\,\mathrm{d}t.
\]
The assumptions imply that for $t\geq K\sqrt{2\ln C}$ we have
\begin{align}\label{eq:boundddd}
\bbP(|\fkt|\geq t)\ \leq\ e^{\ln C-\frac{t^2}{K^2}}\ \leq\ e^{-\frac{t^2}{2K^2}}.
\end{align}
Therefore, 
\begin{align*}\bbE|\fkt|^{\ell}\ &=\ \int_0^{K\sqrt{2\ln C}} \ell t^{\ell-1}\bbP(|\fkt|\geq t)\, \mathrm{d}t+\int_{K\sqrt{2\ln C}}^\infty \ \ell t^{\ell-1}\bbP(|\fkt|\geq t)\, \mathrm{d}t\\ 
&\leq\ K^{\ell}(2\ln C)^{\frac{\ell}{2}}+\int_0^\infty \ell t^{\ell-1}  e^{-\frac{t^2}{2K^2}}\, \mathrm{d}t\\
&\leq\ K^{\ell}(2\ln C)^{\frac{\ell}{2}}+\ell K^{\ell}2^{\frac{\ell}{2}-1} \Gamma\left(\frac{\ell}{2}\right),
\end{align*}
where in the first integral we estimate $\bbP(|\fkt|\geq t)$ by $1$ and to bound the second integral we apply \eqref{eq:boundddd} and then extend the domain of the integration.

Using induction on $\ell$ we bound $\ell2^{\frac{\ell}{2}-1}\Gamma(\ell/2)$ by $
\left(\pi \ell/2\right)^{\frac{\ell}{2}}$, which gives
\[
 \bbE|\fkt|^{\ell}\ \leq\ K^{\ell}\left(\left(2\ln C\right)^{\frac{\ell}{2}}+\left(\frac{\pi \ell}{2}\right)^{\frac{\ell}{2}}\right)\ \leq\ K^{\ell}\left(\left(2\ln C\right)^{\frac{\ell}{2}}+\left(\frac{\pi }{2}\right)^{\frac{\ell}{2}}\right)\ell^{\frac{\ell}{2}}.
\]
Now the claim follows from the inequality comparing the $\ell$-norm and the $1$-norm.

For the last claim, by the same argument as above, we obtain
\begin{align*}
\bbE|\fkt|\ \leq\ K\sqrt{2\ln C}+\int_{K\sqrt{2\ln C}}^{\infty}e^{-\frac{t^2}{2K^2}}\,\mathrm{d}t.
\end{align*}
Finally, for $t\geq K\sqrt{2\ln C}$ we have that $t^2\geq K\sqrt{2\ln C}\,t$ and the claim \eqref{eq:specialbound} follows from
\begin{align*}
\int_{K\sqrt{2\ln C}}^{\infty} e^{-\frac{t^2}{2K^2}}\,\mathrm{d}t\ \leq\ \int_0^{\infty} e^{-\frac{\sqrt{\ln C}}{\sqrt{2}K}t}\,\mathrm{d}t\ =\ \frac{\sqrt{2}K}{\sqrt{\ln C}}.
\end{align*}

\subsection[Proof of (3.5)]{Proof of \eqref{eq:yuprojjacobian}}\label{subsec:yuprojjacobian}

Because of the ``product"-like structure of $\yuproj$ it is enough to consider the case $d=1$,
\begin{align*}
    \yuproj:\bbR^{n-1}&\rightarrow \bbS^{n-1}\\
    \Vector{x}&\mapsto \frac{1}{\sqrt{1+\|\Vector{x}\|_2^2}}\begin{pmatrix}
    1\\\Vector{x}
    \end{pmatrix},
\end{align*}
and to prove
\begin{equation}\label{eq:yuprojjacobian2}
    |\det \diff_{\Vector{x}}\yuproj|\ =\ \left(1+\|\Vector{x}\|_2^2\right)^{-\frac{n}{2}},\quad \Vector{x}\in \bbR^{n-1},
\end{equation}
where $\diff_{\Vector{x}}\yuproj$ is written in some orthonormal bases of $\Tg_{\Vector{x}}\bbR^{n-1}$ and $\Tg_{\yuproj(\Vector{x})} \bbS^{n-1}$.

We fix $\Vector{x} \neq \Vector{0}$, as for $\Vector{x}=\Vector{0}$ the claim follows by continuity. Let us consider an orthonormal basis of $\Tg_{\Vector{x}}\bbR^{n-1}$ given by
\begin{equation*}
    \frac{\Vector{x}}{\|\Vector{x}\|_2},\Vector{v}_1,\ldots,\Vector{v}_{n-2},
\end{equation*}
where $\Vector{v}_1,\dots, \Vector{v}_{n-2}$ form a basis of the orthogonal complement of the line $\bbR \Vector{x}\subset \bbR^{n-1}$. Then vectors 
\begin{equation*}
\frac{1}{\sqrt{1+\|\Vector{x}\|_2^2}}\begin{pmatrix}
-\|\Vector{x}\|_2\\\Vector{x}/\|\Vector{x}\|_2
\end{pmatrix},\begin{pmatrix}
0\\\Vector{v}_1
\end{pmatrix},\ldots,\begin{pmatrix}
0\\\Vector{v}_{n-2}
\end{pmatrix}
\end{equation*}
form an orthogonal basis of $\Tg_{\yuproj(\Vector{x})}\bbS^{n-1}$.
A direct computation shows that
\begin{align*}
        \diff_{\Vector{x}}\yuproj\left(\frac{\Vector{x}}{\|\Vector{x}\|_2}\right)\ &=\ \left.\frac{\mathrm{d}}{\mathrm{d}t}\right|_{t=0}\yuproj\left(\Vector{x}+t\frac{\Vector{x}}{\|\Vector{x}\|_2}\right)\\
        &=\ \left.\frac{\mathrm{d}}{\mathrm{d}t}\right|_{t=0}\left[\frac{1}{\sqrt{1+\|\Vector{x}\|_2^2+2t\|\Vector{x}\|_2+t^2}}\begin{pmatrix}
    1\\\Vector{x}+t\frac{\Vector{x}}{\|\Vector{x}\|_2}
    \end{pmatrix}\right]\\ &=\ \frac{1}{1+\|\Vector{x}\|_2^2}\frac{1}{\sqrt{1+\|\Vector{x}\|_2^2}}\begin{pmatrix}
-\|\Vector{x}\|_2\\\Vector{x}/\|\Vector{x}\|_2
\end{pmatrix},\\
\diff_{\Vector{x}}\yuproj\left(\Vector{v}_i\right)\ &=\ \left.\frac{\mathrm{d}}{\mathrm{d}t}\right|_{t=0}\yuproj(\Vector{x}+t\Vector{v}_i)\ 
=\ \left.\frac{\mathrm{d}}{\mathrm{d}t}\right|_{t=0}\left[\frac{1}{\sqrt{1+\|\Vector{x}\|_2^2+t^2}}\begin{pmatrix}
    1\\\Vector{x}+t\Vector{v}_i
    \end{pmatrix}\right]\\ &=\ \frac{1}{\sqrt{1+\|\Vector{x}\|_2^2}}
    \begin{pmatrix}
0\\\Vector{v}_{i}
\end{pmatrix},\quad i=1,\dots, n-2.
\end{align*}
Finally, the desired formula \eqref{eq:yuprojjacobian2} follows from the fact that in the chosen orthonormal bases of $\Tg_{\Vector{x}}\bbR^{n-1}$ and $\Tg_{\yuproj(\Vector{x})}\bbS^{n-1}$ the differential $\diff_{\Vector{x}}\yuproj$ is given by the matrix
\[
\begin{pmatrix}
\frac{1}{1+\|\Vector{x}\|_2^2}&&& \textrm{{\Large $0$}}\\
&\frac{1}{\sqrt{1+\|\Vector{x}\|_2^2}}&&\\
&&\ddots&\\
\textrm{{\Large $0$}}&&&\frac{1}{\sqrt{1+\|\Vector{x}\|_2^2}}
\end{pmatrix}.
\]

\end{document}